\documentclass[12pt,reqno]{amsart}
\usepackage{amsmath, amsthm, amssymb, stmaryrd}

\advance \topmargin by -\headheight
\advance \topmargin by -\headsep
     
\evensidemargin \oddsidemargin
\newtheorem{theorem}{Theorem}[section]
\newtheorem{lemma}[theorem]{Lemma}
\newtheorem{corollary}[theorem]{Corollary}

\theoremstyle{definition}

\theoremstyle{remark}

\numberwithin{equation}{section}

\newcommand{\mmod}[1]{\,\,(\text{mod}\,\,#1)}

  \def\bff{{\mathbf f}}

\def\bfm{{\mathbf m}}
\def\bfn{{\mathbf n}}

\def\bfu{{\mathbf u}}
\def\bfv{{\mathbf v}}
\def\bfw{{\mathbf w}}
\def\bfx{{\mathbf x}}
\def\bfy{{\mathbf y}}
\def\bfz{{\mathbf z}}

\def\calB{{\mathcal B}} 
\def\calC{{\mathcal C}} 

\def\calD{{\mathcal D}}

\def\calJ{{\mathcal J}}

 \def\Ktil{{\widetilde K}}

\def\calN{{\mathcal N}}

\def\calR{{\mathcal R}}

\def\Gtil{\widetilde G}\def\Itil{\widetilde I}\def\Ktil{\widetilde K}

\def\dbC{{\mathbb C}}\def\dbN{{\mathbb N}}
\def\dbR{{\mathbb R}}
\def\dbZ{{\mathbb Z}}

\def\grf{{\mathfrak f}}\def\grF{{\mathfrak F}}
\def\grG{{\mathfrak G}}

\def\grm{{\mathfrak m}}\def\grM{{\mathfrak M}}
\def\grS{{\mathfrak S}}

\def\alp{{\alpha}} \def\bfalp{{\boldsymbol \alpha}}
\def\bet{{\beta}}  \def\bfbet{{\boldsymbol \beta}}
\def\gam{{\gamma}} \def\Gam{{\Gamma}}
\def\del{{\delta}} \def\Del{{\Delta}}
\def\zet{{\zeta}} \def\bfzet{{\boldsymbol \zeta}} 
\def\bfeta{{\boldsymbol \eta}} \def\Eta{{\mathrm H}}
\def\tet{{\theta}}  
\def\kap{{\kappa}}
\def\lam{{\lambda}}  
\def\bfnu{{\boldsymbol \nu}}
\def\bfxi{{\boldsymbol \xi}}

\def\sig{{\sigma}} \def\Sig{{\Sigma}} \def\bfsig{{\boldsymbol \sig}}
\def\bftau{{\boldsymbol \tau}}
\def\Ups{{\Upsilon}} 
\def\bfpsi{{\boldsymbol \psi}}
\def\ome{{\omega}} \def\Ome{{\Omega}}
\def\d{{\partial}}
\def\eps{\varepsilon}

\def\le{\leqslant} \def\ge{\geqslant}

\def\d{{\,{\rm d}}}

\begin{document}
\title[Vinogradov's mean value theorem]{Vinogradov's mean value theorem\\ via efficient congruencing, II}
\author[Trevor D. Wooley]{Trevor D. Wooley$^*$}
\address{School of Mathematics, University of Bristol, University Walk, Clifton, Bristol BS8 1TW, United Kingdom}
\email{matdw@bristol.ac.uk}
\thanks{$^*$Supported by a Royal Society Wolfson Research Merit Award.}
\date{}
\begin{abstract} We apply the efficient congruencing method to estimate Vinogradov's integral for moments of order $2s$, with $1\le s\le k^2-1$. Thereby, we show that quasi-diagonal behaviour holds when $s=o(k^2)$, we obtain near-optimal estimates for $1\le s\le \frac{1}{4}k^2+k$, and optimal estimates for $s\ge k^2-1$. In this way we come half way to proving the main conjecture in two different directions. There are consequences for estimates of Weyl type, and in several allied applications. Thus, for example, the anticipated asymptotic formula in Waring's problem is established for sums of $s$ $k$th powers of natural numbers whenever $s\ge 2k^2-2k-8$ $(k\ge 6)$.
\end{abstract}
\maketitle

\section{Introduction} Estimates stemming from Vinogradov's mean value theorem deliver bounds for exponential sums of large degree, both in mean and pointwise, beyond the competence of alternate approaches. The ubiquity of such exponential sums in analytic number theory, in the analysis for example of the Riemann zeta function, in Waring's problem, and beyond, accounts for the high profile of Vinogradov's methods in the associated literature. In recent work, we established a version of Vinogradov's mean value theorem which achieves an essentially optimal upper bound with a number of variables only twice the number conjectured to be best possible (see \cite{Woo2011a}). For systems of degree $k$, previous estimates missed such a bound by a factor of order $\log k$. Our earlier approach provides no upper bounds when the number of variables is smaller, precluding the possibility of applications involving the finer features of these mean values. Our goal in this paper is to remedy this deficiency, at the same time strengthening our previous conclusions. It transpires that we are able to come within a hair's breadth of proving the main conjecture concerning Vinogradov's mean value theorem in half of the basic interval of relevant moments. Such developments illustrate the flexibility of the new efficient congruencing method introduced in \cite{Woo2011a}.\par

We now introduce some notation. When $k\in \dbN$ and $\bfalp\in \dbR^k$, define
$$f_k(\bfalp;X)=\sum_{1\le x\le X}e(\alp_1x+\ldots +\alp_kx^k),$$
where $e(z)$ denotes $e^{2\pi iz}$. Our goal is to estimate the mean value
$$J_{s,k}(X)=\oint |f_k(\bfalp;X)|^{2s}\d\bfalp,$$
which by orthogonality counts the solutions of the Diophantine system
$$x_1^j+\ldots +x_s^j=y_1^j+\ldots +y_s^j\quad (1\le j\le k),$$
with $1\le \bfx,\bfy\le X$. Here and elsewhere, we employ the convention that whenever $G:[0,1)^k\rightarrow \dbC$ is measurable, then
$$\oint G(\bfalp)\d\bfalp =\int_{[0,1)^k}G(\bfalp)\d\bfalp .$$
In addition, we make slightly unconventional use of vector notation. Thus, for example, we may write $1\le \bfx\le X$ to denote that $1\le x_i\le X$ $(1\le i\le s)$.\par

We complete the proof of our basic estimate for $J_{s,k}(X)$ in \S8. Here and elsewhere, so far as implicit constants associated with Vinogradov's notation $\ll $ and $\gg $ are concerned, we suppress mention of dependence on $s$, $k$ and $\eps$.

\begin{theorem}\label{theorem1.1} Suppose that $s$ and $k$ are natural numbers with $k\ge 3$ and $s\ge k^2-1$. Then, for each $\eps>0$, one has $J_{s,k}(X)\ll X^{2s-\frac{1}{2}k(k+1)+\eps}$.
\end{theorem}

Prior to the introduction of the efficient congruencing method, conclusions of the type supplied by Theorem \ref{theorem1.1} were available only for $s\ge (1+o(1))k^2\log k$ (see \cite{ACK2004}, \cite{Vin1947}, \cite{Woo1992}, \cite{Woo1996} and earlier work of Hua \cite{Hua1965}). In \cite[Theorem 1.1]{Woo2011a}, meanwhile, we showed that $J_{s,k}(X)\ll X^{2s-\frac{1}{2}k(k+1)+\eps}$ for $s\ge k(k+1)$, and this yields the conclusion of Theorem \ref{theorem1.1} with the condition 
$s\ge k^2-1$ replaced by $s\ge k^2+k$. Our new result is consequently rather sharper than that of \cite{Woo2011a}, which in terms of the constraint on the number of variables already comes within a factor $2$ of the widely held conjecture that $J_{s,k}(X)\ll X^{2s-\frac{1}{2}k(k+1)+\eps}$ for $s\ge \frac{1}{2}k(k+1)$.\par

There are numerous consequences of Theorem \ref{theorem1.1}, with refinements available for estimates of Weyl sums, fractional parts of polynomials, and various Diophantine problems. Since these improvements are modest in scale compared to those made available in our previous work \cite{Woo2011a}, we defer discussion of the bulk of such matters to \S11. For the moment, we choose instead to pursue the more subtle features of the behaviour of the mean value $J_{s,k}(X)$.\par

In order to motivate a discussion of the mean value $J_{s,k}(X)$ for smaller values of $s$, we begin by recalling the lower bound
\begin{equation}\label{1.1}
J_{s,k}(X)\gg X^s+X^{2s-\frac{1}{2}k(k+1)}.
\end{equation}
The closely associated conjectural upper bound
\begin{equation}\label{1.2}
J_{s,k}(X)\ll X^\eps (X^s+X^{2s-\frac{1}{2}k(k+1)}).
\end{equation}
is approximated for $0<s\le \frac{1}{2}k(k+1)$ by an estimate of the shape
\begin{equation}\label{1.3}
J_{s,k}(X)\ll X^{s+\del_{s,k}+\eps},
\end{equation}
provided that $\del_{s,k}\ge 0$ is small. Suppose that (\ref{1.3}) holds for an exponent sequence $\del_{s,k}$ with $\del_{s,k}\rightarrow 0$ as $k\rightarrow \infty$. Then, motivated by our earlier work \cite{Woo1994}, we say that the sequence of mean values $J_{s,k}(X)$ $(k\in \dbN)$ exhibits {\it quasi-diagonal behaviour} for the exponent $s$. It follows from \cite[Theorem 1]{Woo1994} that whenever $s\le k^{3/2}(\log k)^{-1}$, quasi-diagonal behaviour holds for the mean value $J_{s,k}(X)$ in a particularly strong form. Indeed, subject to the latter condition on $s$, the bound (\ref{1.3}) holds for the exponent $\del_{s,k}=\exp(-Ak^3/s^2)$, for a certain positive constant $A$. In \S9 we establish that the mean value $J_{s,k}(X)$ exhibits quasi-diagonal behaviour whenever $s=o(k^2)$.

\begin{theorem}\label{theorem1.2} Suppose that $r$, $k$ and $s$ are natural numbers with $k\ge 3$, $1\le r\le \min\{ k-2,\frac{1}{2}k+1\}$ and $s\le r(k-r+2)$. Put
$$\nu_{r,k}=\frac{r-1}{k-r}.$$
Then for each $\eps>0$, one has the estimate $J_{s,k}(X)\ll X^{s+\nu_{r,k}+\eps}$.
\end{theorem}

In order to compare the strength of the estimate supplied by Theorem \ref{theorem1.2} with that of previous work, it is useful to consider the situation in which $s$ and $k$ are natural numbers with $k$ large and $s\le \frac{1}{4}k^2$, and to put $\lam=s/k^2$. Then the work of Arkhipov and Karatsuba \cite{AK1978} shows that (\ref{1.3}) holds with a permissible exponent $\del_{s,k}$ satisfying $\del_{s,k}\ll \lam^{3/2}k^2$, Tyrina \cite{Tyr1987} obtains $\del_{s,k}\ll \lam^2k^2$, whilst Theorem \ref{theorem1.2} yields the significantly stronger bound $\del_{s,k}\ll \lam$. Notice also that by taking $r=1$ in Theorem \ref{theorem1.2}, one recovers the estimate $J_{k+1,k}(X)\ll X^{k+1+\eps}$ obtained in a slightly sharper form in Hua \cite[Lemma 5.4]{Hua1965}, and sharpened further by Vaughan and Wooley \cite{VW1997}. Finally, by putting $r=[(k+1)/2]$ in Theorem \ref{theorem1.2}, one obtains an attractive estimate simple to state.

\begin{corollary}\label{corollary1.3} Suppose that $s$ and $k$ are natural numbers with $k\ge 4$ and $s\le \frac{1}{4}k^2+k$. Then for each $\eps>0$, one has $J_{s,k}(X)\ll X^{s+1+\eps}$.
\end{corollary}

The estimate supplied by this corollary comes very close indeed to establishing the conjectured estimate (\ref{1.2}) in the interval $1\le s\le \frac{1}{4}k^2+k$. If one were to establish an analogue of Corollary \ref{corollary1.3} in the longer interval $1\le s\le \frac{1}{2}k(k+1)$, then the full conjecture (\ref{1.2}) would essentially follow. In a sense, therefore, Corollary \ref{corollary1.3} comes half way to proving the main conjecture in this subject. When $s\ge k^2-1$, on the other hand, Theorem \ref{theorem1.1} establishes the conjectured bound  (\ref{1.2}). If one were to establish an analogue of Theorem \ref{theorem1.1} for $s\ge \frac{1}{2}k(k+1)$, this would again prove the main conjecture. Thus one comes half way to proving the main conjecture in two different directions.\par

The conclusion of Theorem \ref{theorem1.1} delivers essentially optimal estimates for $J_{s,k}(X)$ when $s\ge k^2-1$. In \S8 we consider the behaviour of $J_{s,k}(X)$ when $s$ is somewhat smaller than $k^2-1$. In this context, it is useful to define the exponent
\begin{equation}\label{1.4}
\Del_{t,k}=\tfrac{1}{2}t(t-1)\Bigl(\frac{k+1}{k-1}\Bigr).
\end{equation}

\begin{theorem}\label{theorem1.4}
Suppose that $s$, $t$ and $k$ are natural numbers with $k\ge 3$, $1\le t\le k-1$ and $s\ge (k-t)(k+1)$. Then for each $\eps>0$, one has
$$J_{s,k}(X)\ll X^{2s-\frac{1}{2}k(k+1)+\Del_{t,k}+\eps}.$$
\end{theorem}

The exponent $\Del_{t,k}$ in the upper bound presented in Theorem \ref{theorem1.4} converges quadratically to zero as $t$ decreases to zero, representing a substantial improvement over the bounds made available by means of linear interpolation via H\"older's inequality. Notice that Theorem \ref{theorem1.1} follows from Theorem \ref{theorem1.4} by simply setting $t=1$.\par

We turn next to applications of our methods in the context of Waring's problem. When $s$ and $k$ are natural numbers, let $R_{s,k}(n)$ denote the number of representations of the natural number $n$ as the sum of $s$ $k$th powers of positive integers. A formal application of the circle method suggests that for $k\ge 3$ and $s\ge k+1$, one should have
\begin{equation}\label{1.5}
R_{s,k}(n)=\frac{\Gam(1+1/k)^s}{\Gam(s/k)}\grS_{s,k}(n)n^{s/k-1}+o(n^{s/k-1}),
\end{equation}
where
$$\grS_{s,k}(n)=\sum_{q=1}^\infty \sum^q_{\substack{a=1\\ (a,q)=1}}\Bigl( q^{-1}\sum_{r=1}^qe(ar^k/q)\Bigr)^se(-na/q).$$
Subject to suitable congruence conditions, one has $1\ll \grS_{s,k}(n)\ll n^\eps$, so that the conjectured relation (\ref{1.5}) represents an honest asymptotic formula. Let $\Gtil(k)$ denote the least integer $t$ with the property that, for all $s\ge t$, and all sufficiently large natural numbers $n$, one has the asymptotic formula (\ref{1.5}). By incorporating the estimates supplied by Theorems \ref{theorem1.1} and \ref{theorem1.4} into our recent work concerning the asymptotic formula in Waring's problem \cite{Woo2011b}, in \S10 we derive the upper bounds for $\Gtil(k)$ contained in the following theorem. We make use here of the notation defined in (\ref{1.4}).

\begin{theorem}\label{theorem1.5}
Let $k$ be a natural number with $k\ge 3$. Then one has
$$\Gtil(k)\le 2k^2-2k+1-\max_{\substack{0\le r\le k-2\\ 2^r\le k^2-k-1}}\left\lceil \frac{2(k-1)(r+1)-2^{r+1}}{k-r}\right\rceil ,$$
and also
$$\Gtil(k)\le 2k^2-1-\underset{2(t-1)(k+1)+m(m-1)<2k^2-2}{\max_{1\le m\le k}\max_{1\le t\le k-1}}\left\lceil \frac{2(k+1)(t-1)-m(m-1)}{1+\Del_{t,k}/m}\right\rceil .$$
\end{theorem}

Two consequences of Theorem \ref{theorem1.5} deserve to be recorded.

\begin{corollary}\label{corollary1.6} When $k$ is a large natural number, one has
$$\Gtil(k)\le 2k^2-k^{4/3}+O(k).$$
\end{corollary}

This conclusion sharpens slightly the bound $\Gtil (k)\le 2k^2-2\left[ (\log k)/(\log 2)\right]$ established recently in \cite[Corollary 1.2]{Woo2011b}.

\begin{corollary}\label{corollary1.7} When $k$ is a natural number with $k\ge 6$, one has
$$\Gtil(k)\le 2k^2-2k-\tet_k,$$
where
$$\tet_k=\begin{cases}8,&\text{when $k=6$,}\\ 9,&\text{when $7\le k\le 13$,}\\ 10,&\text{when $14\le k\le 19$,}\\ 12,&\text{when $k\ge 20$.}\end{cases}$$
\end{corollary}

In particular, one has
$$\Gtil(6)\le 52,\ \Gtil(7)\le 75,\ \Gtil(8)\le 103,\ \Gtil(9)\le 135,\ldots ,\ \Gtil(20)\le 748.$$
For comparison, in \cite[Corollary 1.2]{Woo2011b} we showed that
$$\Gtil(7)\le 86,\, \Gtil(8)\le 117,\ \Gtil(9)\le 151,\ldots ,\ \Gtil(20)\le 789.$$
Work preceding the introduction of efficient congruencing delivered substantially weaker conclusions. Thus, for smaller values of $k$, by using a refinement of an earlier method of Heath-Brown \cite{HB1988}, it was shown by Boklan \cite{Bok1994} that
$$\Gtil(6)\le 56,\ \Gtil(7)\le 112,\ \Gtil(8)\le 224.$$
For large values of $k$, meanwhile, one had the work of Ford \cite{For1995}. Together with refinements for intermediate values of $k$ due to Parsell \cite{Par2009} and Boklan and Wooley \cite{BW2011}, this delivered the bounds
$$\Gtil(9)\le 365,\ldots ,\Gtil(20)\le 2534,\quad \text{and}\quad \Gtil(k)\le k^2(\log k+\log \log k+O(1)).$$
We note that the methods underlying the proof of Theorem \ref{theorem1.5} fail by $\eps$ to deliver the bound $\Gtil(5)\le 32$ established by Vaughan \cite{Vau1986}. Thus, our methods come within a whisker of achieving useful conclusions even for $k=5$.\par

We establish Theorems \ref{theorem1.1}, \ref{theorem1.2} and \ref{theorem1.4} by means of the {\it efficient congruencing} method introduced in our earlier work \cite{Woo2011a}. A sketch of the method is provided in \cite[\S2]{Woo2011a}, and the reader may find this a helpful guide when it comes to understanding the basic plan of attack in this paper. It is a notable feature of this earlier work that, when successful for a given choice of $s$, the method yields a bound of the shape $J_{s,k}(X)\ll X^{2s-\frac{1}{2}k(k+1)+\eps}$, within a factor $X^\eps$ of the sharpest bound conjectured to hold. In this paper we adapt the efficient congruencing method so as to obtain weaker bounds of the shape $J_{s,k}(X)\ll X^{2s-\kap(s,k)+\eps}$, wherein $\kap(s,k)<\frac{1}{2}k(k+1)$. Although this advance may seem to provide only modest additional flexibility, it is neither trivial nor inconsequential. Further differences will be encountered from \cite{Woo2011a} in the handling of auxiliary congruences, and in particular linear congruence information is more efficiently handled implicitly within the main congruencing process.\par

We organise this paper as follows. In \S2 we invest in some preliminary manoeuvres and introduce notation that facilitates what follows. Estimates for auxiliary congruences are established in \S3, and in \S4 we perform the conditioning of variables that permits non-singularity constraints to be imposed on the variables where needed. The efficient congruencing process is described in two stages. In \S5 we perform the efficient congruencing step itself. Then, following discussion of an initial pre-congruencing step in \S6, we advance in \S7 to extract from the conclusions of \S5 a formulation suitable for iterating the efficient congruencing process. We now come to the iterative relations, and these differ according to the variable regime of interest. In \S\S8 and 9 we establish, respectively, Theorems \ref{theorem1.1} and \ref{theorem1.4}, and Theorem \ref{theorem1.2}. Then in \S10, we discuss the asymptotic formula in Waring's problem, proving Theorem \ref{theorem1.5} and its corollaries. Finally, in \S11, we consider several further consequences of our new estimates. Here we highlight improvements in estimates of Weyl type, the distribution of polynomials modulo $1$, Tarry's problem, and an estimate of Croot and Hart related to the sum-product theorem.

\section{Preliminaries and infrastructure}  Our objective in this section is to introduce such notation and preliminary estimates as are needed to describe the infrastructure of the repeated efficient congruencing process. In what follows, the letter $k$ denotes a fixed integer exceeding $2$, the letter $s$ will be a positive integer, and $\eps$ denotes a sufficiently small positive number. The basic parameter occurring in our asymptotic estimates is $X$, a large real number depending at most on $k$, $s$ and $\eps$, unless otherwise indicated. In an effort to simplify our exposition, we adopt the following convention concerning the number $\eps$. Whenever $\eps$ appears in a statement, either implicitly or explicitly, we assert that the statement holds for each $\eps>0$. Note that the ``value'' of $\eps$ may consequently change from statement to statement. We are relatively cavalier concerning the use of vector notation. In particular, we may write $\bfz\equiv \bfw\pmod{p}$ to denote that $z_i\equiv w_i\pmod{p}$ $(1\le i\le t)$, or even $\bfz\equiv \xi\pmod{p}$ to denote that $z_i\equiv \xi\pmod{p}$ $(1\le i\le t)$. Finally, throughout \S\S2--9, we consider the integer $k$ to be fixed, and we therefore abbreviate $J_{s,k}(X)$ to $J_s(X)$, and likewise $f_k(\bfalp;X)$ to $f(\bfalp;X)$, without further comment.\par

Our attention is focused on the mean value $J_s(X)$ where, for the moment, we think of $s$ as being an arbitrary natural number. We refer to the exponent $\lam_s$ as {\it permissible} when, for each positive number $\eps$, and for any real number $X$ sufficiently large in terms of $s$, $k$ and $\eps$, one has $J_s(X)\ll X^{\lam_s+\eps}$. Define $\lam_s^*$ to be the infimum of the set of exponents $\lam_s$ permissible for $s$ and $k$. In view of the conjectured upper bound (\ref{1.2}) and the corresponding lower bound (\ref{1.1}), we expect that for each natural number $s$, one should have
$$\lam_s^*=\max\{s,2s-\tfrac{1}{2}k(k+1)\}.$$
In our earlier work \cite{Woo2011a}, we sought to establish that $\lam_s^*=2s-\frac{1}{2}k(k+1)$ with $s$ as small as possible, and indeed we established such for $s\ge k(k+1)$. In present circumstances we are less ambitious, though we ultimately prove more. With this in mind, we take $\kap_s=\kap(s,k)$ to be a positive parameter to be chosen in due course, but satisfying $\kap_s\le \max\{s,\frac{1}{2}k(k+1)\}$. In addition, we define $\eta_s=\eta_s(\kap_s,k)$ by putting $\eta_s=\lam_s^*-2s+\kap_s$. Thus, whenever $X$ is sufficiently large in terms of $s$, $k$ and $\eps$, one has
\begin{equation}\label{2.1}
J_s(X)\ll X^{\lam_s^*+\eps},
\end{equation}
where
\begin{equation}\label{2.2}
\lam_s^*=2s-\kap_s+\eta_s.
\end{equation}

\par Rather than investigate the sequence of exponents $\lam_s^*$ directly, it is more convenient instead to fix a natural number $r$ with
\begin{equation}\label{2.3}
1\le r\le k-1,
\end{equation}
and then seek to bound $\lam_{s+r}^*$. By choosing $\kap_{s+r}$ carefully in terms of $s$, we are able to apply the efficient congruencing process to show that $\eta_{s+r}$ may be taken to be an arbitrarily small positive number, and thereby we demonstrate that in fact $\lam_{s+r}^*\le 2s+2r-\kap_{s+r}$. We determine $\kap_{s+r}$ in terms of $s$ and $k$ by means of the parameter $r$ as follows. Fix natural numbers $s$ and $s_0$ with $s\ge s_0$, and write
\begin{equation}\label{2.4}
\rho=k-r+1.
\end{equation}
When it comes to proving Theorem \ref{theorem1.2} we take
\begin{equation}\label{2.5}
s_0=r\rho\quad \text{and}\quad \kap_{s+r}=s_0+r-\frac{r-1}{k-r},
\end{equation}
and for the proof of Theorem \ref{theorem1.4} we take
\begin{equation}\label{2.6}
s_0=rk\quad \text{and}\quad \kap_{s+r}=(rk-\tfrac{1}{2}r(r+1))\left(\frac{k+1}{k-1}\right).
\end{equation}
Our goal is to show that $\lam_{s+r}^*\le 2(s+r)-\kap_{s+r}$, and so we suppose by way of contradiction that in fact
$$\lam_{s+r}^*=2(s+r)-\kap_{s+r}+\eta_{s+r},$$
with $\eta_{s+r}>0$.\par

Let $\del$ be a small positive number to be chosen shortly. In view of the infimal definition of $\lam_{s+r}^*$, there exists a sequence of natural numbers $(X_n)_{n=1}^\infty$, tending to infinity, with the property that
\begin{equation}\label{2.7}
J_{s+r}(X_n)>X_n^{\lam_{s+r}^*-\del}\quad (n\in \dbN).
\end{equation}
Provided that $X_n$ is sufficiently large, it follows from (\ref{2.1}) that for $X_n^{\del^2}<Y\le X_n$, one has the corresponding upper bound
\begin{equation}\label{2.8}
J_{s+r}(Y)<Y^{\lam_{s+r}^*+\del}.
\end{equation}
Notice that since $s\ge s_0$, the trivial inequality $|f(\bfalp;X)|\le X$ yields the upper bound
$$J_{s+r}(X)\le X^{2(s-s_0)}\oint |f(\bfalp;X)|^{2s_0+2r}\d\bfalp =X^{2(s-s_0)}J_{s_0+r}(X).$$
Consequently, one has $\eta_{s+r}\le \eta_{s_0+r}$, and so we are at liberty to restrict attention to the special case $s=s_0$. Since $s_0$ is a multiple of $r$, we consider a fixed natural number $u$ with $u\ge s_0/r$, and put $s=ru$. We keep in play the general case $s\ge s_0$ until the final stages of our argument, the better to illuminate the underlying ideas. Finally, we take $N$ to be a natural number sufficiently large in terms of $s$, $k$ and $r$. In our proofs of Theorems \ref{theorem1.2} and \ref{theorem1.4} we put
\begin{equation}\label{2.9}
\tet=N^{-1/2}(r/s)^{N+2}
\end{equation}
and fix $\del$ to be a positive number with $\del<(Ns)^{-3N}$, so that $\del$ is small compared to $\tet$. We now take a fixed element $X=X_n$ of the sequence $(X_n)$, which we may assume to be sufficiently large in terms of $s$, $k$, $r$, $N$ and $\del$, and put $M=X^\tet$. In particular, we have $X^\del<M^{1/N}$.\par

Let $p$ be a fixed prime number with $M<p\le 2M$ to be chosen in due course. That such a prime exists is a consequence of the Prime Number Theorem. When $c$ and $\xi$ are non-negative integers, and $\bfalp \in [0,1)^k$, define
\begin{equation}\label{2.10}
\grf_c(\bfalp;\xi)=\sum_{\substack{1\le x\le X\\ x\equiv \xi\mmod{p^c}}}e(\psi(x;\bfalp)),
\end{equation}
where
$$\psi(x;\bfalp)=\alp_1x+\alp_2x^2+\ldots +\alp_kx^k.$$
As in \cite{Woo2011a}, we must consider well-conditioned tuples of integers belonging to distinct congruence classes modulo a suitable power of $p$, though now we must proceed in greater generality. Denote by $\Xi_c^r(\xi)$ the set of $r$-tuples $(\xi_1,\ldots ,\xi_r)$, with
$$1\le \xi_i\le p^{c+1}\quad \text{and}\quad \xi_i\equiv \xi\pmod{p^c}\quad (1\le i\le r),$$
and satisfying the property that $\xi_i\equiv \xi_j\pmod{p^{c+1}}$ for no $i$ and $j$ with $1\le i<j\le r$. In addition, write $\Sig_r=\{1,-1\}^r$, and consider an element $\bfsig$ of $\Sig_r$. We then define
\begin{equation}\label{2.11}
\grF_c^\bfsig(\bfalp;\xi)=\sum_{\bfxi\in \Xi_c^r(\xi)}\prod_{i=1}^r\grf_{c+1}(\sig_i\bfalp;\xi_i).
\end{equation}
Notice that we have suppressed mention of the parameter $r$ in our notation for the exponential sum $\grF_c^\bfsig(\bfalp;\xi)$, based on the premise that any possible confusion should be easily avoided.\par

Two mixed mean values are important within our arguments. First, when $a$ and $b$ are positive integers and $\bfsig\in \Sig_r$, we define
\begin{equation}\label{2.12}
I_{a,b}^\bfsig(X;\xi,\eta)=\oint |\grF_a^\bfsig(\bfalp;\xi)^2\grf_b(\bfalp;\eta)^{2s}|\d\bfalp
\end{equation}
and
\begin{equation}\label{2.13}
K^{\bfsig,\bftau}_{a,b}(X;\xi,\eta)=\oint |\grF_a^\bfsig(\bfalp;\xi)^2\grF_b^\bftau (\bfalp;\eta)^{2u}|\d \bfalp .
\end{equation}
It is convenient then to put
\begin{equation}\label{2.14}
I_{a,b}(X)=\max_{1\le \xi\le p^a}\max_{\substack{1\le \eta\le p^b\\ \eta\not\equiv \xi\mmod{p}}}\max_{\bfsig \in \Sig_r}I^\bfsig_{a,b}(X;\xi,\eta)
\end{equation}
and
\begin{equation}\label{2.15}
K_{a,b}(X)=\max_{1\le \xi\le p^a}\max_{\substack{1\le \eta\le p^b\\ \eta\not\equiv \xi\mmod{p}}}\max_{\bfsig,\bftau\in \Sig_r}K_{a,b}^{\bfsig,\bftau}(X;\xi,\eta).
\end{equation}
The implicit dependence of these mean values on our choice of $p$ will ultimately be rendered irrelevant, since we fix $p$ in the pre-congruencing step described in \S6, following the proof of Lemma \ref{lemma6.1}. We defer the definition of $K_{0,b}(X)$ to \S6, since there are technical complications better avoided at this stage.\par

As in \cite{Woo2011a}, our arguments are simplified by making transparent the relationship between mean values and their anticipated magnitudes. In this context, we define $\llbracket J_{s+r}(X)\rrbracket$ by means of the relation
\begin{equation}\label{2.16}
J_{s+r}(X)=X^{2s+2r-\kap_{s+r}}\llbracket J_{s+r}(X)\rrbracket ,
\end{equation}
and when $0\le a<b$, we define $\llbracket K_{a,b}(X)\rrbracket$ by means of the relation
\begin{equation}\label{2.17}
K_{a,b}(X)=(X/M^b)^{2s}(X/M^a)^{2r-\kap_{s+r}}\llbracket K_{a,b}(X)\rrbracket.
\end{equation}
The lower bound (\ref{2.7}) may now be written
\begin{equation}\label{2.18}
\llbracket J_{s+r}(X)\rrbracket >X^{\eta_{s+r}-\del}.
\end{equation}

\par We finish this section by recalling an estimate from \cite{Woo2011a} that encapsulates the translation-dilation invariance of the Diophantine system underlying the mean value $J_s(X)$.

\begin{lemma}\label{lemma2.1}
Suppose that $c$ is a non-negative integer with $c\tet\le 1$. Then for each natural number $t$, one has
$$\max_{1\le \xi\le p^c}\oint |\grf_c(\bfalp;\xi)|^{2t}\,d\bfalp \ll_t J_t(X/M^c).$$
\end{lemma}

\begin{proof} This is \cite[Lemma 3.1]{Woo2011a}.
\end{proof}

\section{Auxiliary systems of congruences} Following the pattern established in our initial work \cite{Woo2011a} concerning efficient congruencing, we begin the main thrust of our analysis with a discussion of the congruences that play a critical role in what follows. Two basic arrangements of the congruencing idea are required, and these we handle in separate lemmata. We prepare the ground first with some notation.\par

Recall that $r$ is an integer with $1\le r\le k-1$. When $a$ and $b$ are integers with $1\le a<b$, and $\bfsig\in \Sig_r$, we denote by $\calB_{a,b}^{\bfsig,r}(\bfm;\xi,\eta)$ the set of solutions of the system of congruences
\begin{equation}\label{3.1}
\sum_{i=1}^r\sig_i(z_i-\eta)^j\equiv m_j\mmod{p^{jb}}\quad (1\le j\le k),
\end{equation}
with $1\le \bfz\le p^{kb}$ and $\bfz\equiv \bfxi\pmod{p^{a+1}}$ for some $\bfxi\in \Xi_a^r(\xi)$. We define an equivalence relation $\calR(\lam)$ on integral $r$-tuples by declaring the $r$-tuples $\bfx$ and $\bfy$ to be $\calR(\lam)$-equivalent when $\bfx\equiv \bfy\pmod{p^\lam}$. We then write $\calC_{a,b}^{\bfsig,r,h}(\bfm;\xi,\eta)$ for the set of $\calR(hb)$-equivalence classes of $\calB_{a,b}^{\bfsig,r}(\bfm;\xi,\eta)$, and we define $B_{a,b}^{r,h}(p)$ by putting
\begin{equation}\label{3.2}
B_{a,b}^{r,h}(p)=\max_{1\le \xi\le p^a}\max_{\substack{1\le \eta\le p^b\\ \eta\not\equiv \xi\mmod{p}}}\max_{\bfsig \in \Sig_r}\max_{1\le \bfm\le p^{kb}}\text{card}(\calC_{a,b}^{\bfsig ,r,h}(\bfm;\xi,\eta)).
\end{equation}
On considering representatives of the $\calR(hb)$-equivalence classes of the set $\calB_{a,b}^{\bfsig,r}(\bfm;\xi,\eta)$, of course, we may interpret $\calC_{a,b}^{\bfsig,r,h}(\bfm;\xi,\eta)$ via the relation $$\calC_{a,b}^{\bfsig,r,h}(\bfm;\xi,\eta)=\{ \bfx\mmod{p^{hb}}:\bfx\in \calB_{a,b}^{\bfsig,r}(\bfm;\xi,\eta)\}.$$

\par When $a=0$ we modify these definitions, so that $\calB_{0,b}^{\bfsig,r}(\bfm;\xi,\eta)$ denotes the set of solutions of the system of congruences (\ref{3.1}) with $1\le \bfz\le p^{kb}$ and $\bfz\equiv \bfxi\mmod{p}$ for some $\bfxi\in \Xi_0^r(\xi)$, and for which in addition one has $\bfz\not\equiv \eta\mmod{p}$. As in the previous case, we write $\calC_{0,b}^{\bfsig,r,h}(\bfm;\xi,\eta)$ for the set of $\calR(hb)$-equivalence classes of $\calB_{0,b}^{\bfsig,r}(\bfm;\xi,\eta)$, but we define $B_{0,b}^{r,h}(p)$ by putting
\begin{equation}\label{3.3}
B_{0,b}^{r,h}(p)=\max_{1\le \eta\le p^b}\max_{\bfsig \in \Sig_r}\max_{1\le \bfm\le p^{kb}}\text{card}(\calC_{0,b}^{\bfsig ,r,h}(\bfm;0,\eta)).
\end{equation}
We note that the choice of $\xi$ in this situation with $a=0$ is irrelevant, since one has $\xi\equiv 0\pmod{p^a}$ for all integers $\xi$. However, it is notationally convenient to preserve the similarity with the corresponding notation relevant to the situation with $a\ge 1$.\par

We aim to estimate $B_{a,b}^{r,h}(p)$ by exploiting the underlying non-singularity of the solution set via Hensel's lemma. A suitable version of the latter lifting process is implicitly contained within the following lemma.

\begin{lemma}\label{lemma3.1}
Let $f_1,\ldots ,f_d$ be polynomials in $\dbZ[x_1,\ldots ,x_d]$ with respective degrees $k_1,\ldots ,k_d$, and write
$$J(\bff;\bfx)=\mathrm{det}\left( \frac{\partial f_j}{\partial x_i}(\bfx)\right)_{1\le i,j\le d}.$$
When $\varpi$ is a prime number, and $l$ is a natural number, let $\calN(\bff;\varpi^l)$ denote the number of solutions of the simultaneous congruences
$$f_j(x_1,\ldots ,x_d)\equiv 0\pmod{\varpi^l}\quad (1\le j\le d),$$
with $1\le x_i\le \varpi^l$ $(1\le i\le d)$ and $(J(\bff;\bfx),\varpi)=1$. Then $\calN(\bff;\varpi^l)\le k_1\cdots k_d$.
\end{lemma}

We prepare a second auxiliary lemma in order to facilitate discussion of a certain argument involving elimination of terms amongst systems of polynomials. In this context, we adopt the convention that when $l$ and $m$ are natural numbers with $l>m$, then the binomial coefficient $\binom{m}{l}$ is zero.

\begin{lemma}\label{lemma3.2} Let $\alp$ and $\bet$ be natural numbers. Then there exist integers $c_l$ $(\alp\le l\le \alp+\bet)$ and $d_m$ $(\bet\le m\le \alp+\bet)$, depending at most on $\alp$ and $\bet$, and with $d_\bet\ne 0$, for which one has the polynomial identity
\begin{equation}\label{3.4}
c_\alp+\sum_{l=1}^\bet c_{\alp+l}(x+1)^{\alp+l}=\sum_{m=\bet}^{\alp+\bet}d_mx^m.
\end{equation}
\end{lemma}

\begin{proof} Consider the system of equations
\begin{equation}\label{3.5}
\sum_{l=1}^\bet \binom{\alp+l}{m}y_{\alp+l}=\mu_m\quad (1\le m\le \bet),
\end{equation}
in which $\mu_m$ is $0$ when $1\le m<\bet$, and $1$ when $m=\bet$. By comparing coefficients of powers of $x$ on left and right hand sides of (\ref{3.4}), we see that the conclusion of the lemma follows provided that the system of linear equations (\ref{3.5}) admits a rational solution $\bfy$. Indeed, given such a solution, on taking $d_\bet$ to be the least common multiple of the denominators of $y_{\alp+l}$ $(1\le l\le \bet)$, one finds that there exist integers $c_\alp$ and $d_m$ $(\bet<m\le \alp+\bet)$ for which the identity (\ref{3.4}) holds with $c_{\alp+l}=d_\bet y_{\alp+l}$ $(1\le l\le \bet)$.\par

We now demonstrate that the system (\ref{3.5}) does indeed possess a rational solution. When $1\le m\le \bet$, write
$$\psi_m(t)=t(t-1)\ldots (t-m+1).$$
Then on multiplying the equations indexed by $m$ in (\ref{3.5}) through by $m!$, one finds that this system is equivalent to
$$\sum_{l=1}^\bet \psi_m(\alp+l)y_{\alp+l}=\bet!\mu_m\quad (1\le m\le \bet).$$
Hence, on taking linear combinations of these equations, one discerns that (\ref{3.5}) is in turn equivalent to the system of equations
\begin{equation}\label{3.6}
\sum_{l=1}^\bet (\alp+l)^my_{\alp+l}=\bet!\mu_m\quad (1\le m\le \bet).
\end{equation}
The matrix of coefficients of this system has determinant equal to the Vandermonde determinant
$$\mathrm{det}\left( (\alp+l)^m\right)_{1\le l,m\le \bet}=\prod_{1\le l<m\le \bet}\left( (\alp+l)-(\alp+m)\right) \ne 0,$$
and hence is invertible. We therefore deduce by means of Cramer's rule that the system (\ref{3.6}) possesses a rational solution depending only on its coefficients, thus depending only on $\alp$ and $\bet$. The same is consequently true of the equivalent system (\ref{3.5}). In view of the discussion of the first paragraph, this suffices to complete the proof of the lemma.
\end{proof}

Our first bound for $B_{a,b}^{r,h}(p)$ addresses the scenario in which $r<k$, but $h=k$. In a sense, this situation is one in which we discard the $k-r$ congruences of smallest modulus $p^{jb}$ $(1\le j\le k-r)$ but nonetheless aim to lift solutions to the maximum modulus $p^{kb}$. This lemma must be prepared in two variants, one for the case $a\ge 1$ and a second for $a=0$. Before announcing the lemma and its proof, we emphasise that throughout \S\S3-9, we assume $r$ to be constrained by (\ref{2.3}), and define $\rho$ by means of (\ref{2.4}).

\begin{lemma}\label{lemma3.3} Suppose that $a$ and $b$ are integers with $1\le a<b$. Then
$$B_{a,b}^{r,k}(p)\le k!p^{\frac{1}{2}r(r-1)(a+b)}.$$
\end{lemma}

\begin{proof} Consider fixed integers $a$ and $b$ with $1\le a<b$, a fixed $r$-tuple $\bfsig \in \Sig_r$, and fixed integers $\xi$ and $\eta$ with $1\le \xi\le p^a$, $1\le \eta\le p^b$ and $\eta\not\equiv \xi\mmod{p}$. We denote by $\calD_1(\bfn)$ the set of $\calR(kb)$-equivalence classes of solutions of the system of congruences
\begin{equation}\label{3.7}
\sum_{i=1}^r\sig_i(z_i-\eta)^j\equiv n_j\mmod{p^{kb}}\quad (\rho\le j\le k),
\end{equation}
with $1\le \bfz\le p^{kb}$ and $\bfz\equiv \bfxi\mmod{p^{a+1}}$ for some $\bfxi\in \Xi_a^r(\xi)$. Given a fixed integral $r$-tuple $\bfm$, the number of $r$-tuples $\bfn$ with $1\le \bfn\le p^{kb}$ for which
$$n_j\equiv m_j\mmod{p^{jb}}\quad (k-r+1\le j\le k)$$
is equal to
$$\prod_{j=k-r+1}^kp^{(k-j)b}=(p^b)^{\frac{1}{2}r(r-1)}.$$
Consequently, it follows from (\ref{3.1}) that
\begin{align}
\text{card}(\calC_{a,b}^{\bfsig,r,k}(\bfm;\xi,\eta))&\le \sum_{\substack{1\le n_\rho\le p^{kb}\\ n_\rho\equiv m_\rho\mmod{p^{\rho b}}}}\ldots \sum_{\substack{1\le n_k\le p^{kb}\\ n_k\equiv m_k\mmod{p^{kb}}}}\text{card}(\calD_1(\bfn))\notag \\
&\le (p^b)^{\frac{1}{2}r(r-1)}\max_{1\le \bfn\le p^{kb}}\text{card}(\calD_1(\bfn)).\label{3.8}
\end{align}

\par We next rewrite each variable $z_i$ in the shape $z_i=p^ay_i+\xi$. In view of the hypothesis that $\bfz\equiv \bfxi\mmod{p^{a+1}}$ for some $\bfxi\in \Xi_a(\xi)$, the $r$-tuple $\bfy$ necessarily satisfies the property that \begin{equation}\label{3.9}
y_i\not\equiv y_m\mmod{p}\quad (1\le i<m\le r).
\end{equation}
Write $\zet=\xi-\eta$, and note that the constraint $\eta\not\equiv \xi\mmod{p}$ ensures that $p\nmid \zet$. It follows that there exists a multiplicative inverse of $\zet$ modulo $p^{kb}$, and we denote this by $\zet^{-1}$. Then we deduce from (\ref{3.7}) that $\text{card}(\calD_1(\bfn))$ is bounded above by the number of $\calR(kb-a)$-equivalence classes of solutions of the system of congruences
\begin{equation}\label{3.10}
\sum_{i=1}^r\sig_i(p^ay_i\zet^{-1}+1)^j\equiv n_j(\zet^{-1})^j\mmod{p^{kb}}\quad (\rho\le j\le k),
\end{equation}
with $1\le \bfy\le p^{kb-a}$ satisfying (\ref{3.9}). Let $\bfy=\bfw$ be any solution of the system (\ref{3.10}), if indeed such a solution exists. Then we find that all other solutions $\bfy$ satisfy the system of congruences
\begin{equation}\label{3.11}
\sum_{i=1}^r\sig_i\left((p^ay_i\zet^{-1}+1)^j-(p^aw_i\zet^{-1}+1)^j\right)\equiv 0\mmod{p^{kb}}\quad (\rho\le j\le k).
\end{equation}

\par It is at this point that we make use of Lemma \ref{lemma3.2}. Consider an index $j$ with $\rho\le j\le k$, and apply the latter lemma with $\alp=\rho-1$ and $\bet=j-\rho+1$. We deduce that there exist integers $c_{jl}$ $(\rho-1\le l\le j)$ and $d_{jm}$ $(j-\rho+1\le m\le j)$, depending at most on $j$ and $k$, and with $d_{j,j-\rho+1}\ne 0$, for which one has the polynomial identity
\begin{equation}\label{3.12}
c_{j,\rho-1}+\sum_{l=\rho}^jc_{jl}(x+1)^l=\sum_{m=j-\rho+1}^jd_{jm}x^m.
\end{equation}
Since we may assume $p$ to be sufficiently large in terms of $d_{j,j-\rho+1}$, moreover, there is no loss of generality in supposing that $p\nmid d_{j,j-\rho+1}$. Then by multiplying the equation (\ref{3.12}) through by the multiplicative inverse of $d_{j,j-\rho+1}$ modulo $p^{kb}$, we see that there is no loss in supposing that $d_{j,j-\rho+1}\equiv 1\mmod{p^{kb}}$. By taking suitable linear combinations of the congruences comprising (\ref{3.11}), we thus infer that any solution of this system satisfies
$$(\zet^{-1}p^a)^{j-\rho+1}\sum_{i=1}^r\sig_i(\psi_j(y_i)-\psi_j(w_i))\equiv 0\mmod{p^{kb}}\quad (\rho\le j\le k),$$
in which we have written
\begin{equation}\label{3.13}
\psi_j(z)=z^{j-\rho+1}+\sum_{m=j-\rho+2}^jd_{jm}(\zet^{-1}p^a)^{m-j+\rho-1}z^m.
\end{equation}
Note here, in particular, that
\begin{equation}\label{3.14}
\psi_j(z)\equiv z^{j-\rho+1}\mmod{p}.
\end{equation}

\par Denote by $\calD_2(\bfu)$ the set of $\calR(kb-a)$-equivalence classes of solutions of the system of congruences
$$\sum_{i=1}^r\sig_i\psi_j(y_i)\equiv u_j\mmod{p^{kb-(j-\rho+1)a}}\quad (\rho\le j\le k),$$
with $1\le \bfy\le p^{kb-a}$ satisfying (\ref{3.9}). Then we have shown thus far that
\begin{equation}\label{3.15}
\text{card}(\calD_1(\bfn))\le \max_{1\le \bfu\le p^{kb}}\text{card}(\calD_2(\bfu)).
\end{equation}
Let $\calD_3(\bfv)$ denote the set of $\calR(kb-a)$-equivalence classes of solutions of the system
$$\sum_{i=1}^r\sig_i\psi_j(y_i)\equiv v_j\mmod{p^{kb-a}}\quad (\rho\le j\le k),$$
with $1\le \bfy\le p^{kb-a}$ satisfying (\ref{3.9}). Then
\begin{align}
\text{card}(\calD_2(\bfu))&\le \sum_{\substack{1\le v_\rho\le p^{kb-a}\\ v_\rho\equiv u_\rho\mmod{p^{kb-a}}}}\ldots \sum_{\substack{1\le v_k\le p^{kb-a}\\ v_k\equiv u_k\mmod{p^{kb-ra}}}}\text{card}(\calD_3(\bfv))\notag \\
&\le (p^a)^{\frac{1}{2}r(r-1)}\max_{1\le \bfv\le p^{kb-a}}\text{card}(\calD_3(\bfv)).\label{3.16}
\end{align}

\par Define the determinant
\begin{equation}\label{3.17}
J(\bfpsi;\bfx)=\mathrm{det}\left(\sig_i\psi'_{\rho+l-1}(x_i)\right)_{1\le i,l\le r}.
\end{equation}
We claim that when $y_i\equiv y_m\mmod{p}$ for no $i$ and $m$ with $1\le i<m\le r$, then $(J(\bfpsi;\bfy),p)=1$. Temporarily assuming the validity of this claim, we deduce from Lemma \ref{lemma3.1} that $\text{card}(\calD_3(\bfv))\le \rho(\rho+1)\cdots k\le k!$. In view of the definition (\ref{3.2}), the conclusion of the lemma follows at once from (\ref{3.8}), (\ref{3.15}) and (\ref{3.16}).\par

In order to confirm the validity of our claim concerning the Jacobian determinant, we begin by observing that (\ref{3.14}) implies that
$$\sig_i\psi'_{\rho+l-1}(y_i)\equiv \sig_ily_i^{l-1}\mmod{p}.$$
Since we have supposed $p$ to be large compared to $k$, we find that $p|J(\bfpsi;\bfy)$ if and only if
$$\text{det}(y_i^{l-1})_{1\le i,l\le r}\equiv 0\mmod{p}.$$
But by hypothesis we have $y_i\equiv y_m\mmod{p}$ for no $i$ and $m$ with $1\le i<m\le r$, and so it follows that
$$\text{det}(y_i^{l-1})_{1\le i,l\le r}=\prod_{1\le i<m\le r}(y_i-y_m)\not\equiv 0\mmod{p}.$$
We are therefore forced to conclude that $p\nmid J(\bfpsi;\bfy)$, thereby confirming the validity of our earlier claim, and completing the proof of the lemma.
\end{proof}
 
A variant of Lemma \ref{lemma3.3} supplies an analogue applicable in the case $a=0$.

\begin{lemma}\label{lemma3.4}
Suppose that $b$ is an integer with $b\ge 1$. Then
$$B_{0,b}^{r,k}(p)\le k!p^{\frac{1}{2}r(r-1)b}.$$
\end{lemma}

\begin{proof} Consider a fixed integer $b$ with $b\ge 1$, a fixed $r$-tuple $\bfsig\in \Sig_r$, and a fixed integer $\eta$ with $1\le \eta\le p^b$. We denote by $\calD_1(\bfn;\eta)$ the set of $\calR (kb)$-equivalence classes of solutions of the system of congruences (\ref{3.7}) with $1\le \bfz\le p^{kb}$ and $\bfz\equiv \bfxi\mmod{p}$ for some $\bfxi\in \Xi_0^r(0)$, and for which in addition $\bfz\not\equiv \eta\mmod{p}$. Then as in the opening paragraph of the proof of Lemma \ref{lemma3.3}, it follows from (\ref{3.1}) that
\begin{equation}\label{3.18}
\mathrm{card}(\calC_{0,b}^{\bfsig,r,k}(\bfm;0,\eta))\le (p^b)^{\frac{1}{2}r(r-1)}\max_{1\le \bfn\le p^{kb}}\mathrm{card}(\calD_1(\bfn;\eta)).
\end{equation}

\par But $\calD_1(\bfn;\eta)=\calD_1(\bfn;0)$, and $\calD_1(\bfn;0)$ counts the solutions of the system of congruences
$$\sum_{i=1}^r\sig_iy_i^j\equiv n_j\mmod{p^{kb}}\quad (\rho\le j\le k),$$
with $1\le \bfy\le p^{kb}$ satisfying (\ref{3.9}), and in addition $p\nmid y_i$ $(1\le i\le r)$. Write
\begin{equation}\label{3.19}
J(\bfy)=\mathrm{det}\left( (\rho+j-1)\sig_iy_i^{\rho+j-2}\right)_{1\le i,j\le r}.
\end{equation}
Then since $p$ is large compared to $k$, we find that $p|J(\bfy)$ if and only if
$$(y_1\ldots y_r)^{\rho-1}\mathrm{det}\left( y_i^{j-1}\right)_{1\le i,j\le r}\equiv 0\mmod{p}.$$
But by hypothesis we have $(y_1\ldots y_r,p)=1$ and $y_i\equiv y_j\mmod{p}$ for no $i$ and $j$ with $1\le i<j\le r$, and so it follows that
$$(y_1\ldots y_r)^{\rho-1}\mathrm{det}\left( y_i^{j-1}\right)_{1\le i,j\le r}=(y_1\ldots y_r)^{\rho-1}\prod_{1\le i<j\le r}(y_i-y_j)\not\equiv 0\mmod{p}.$$
We therefore deduce from Lemma \ref{lemma3.1} that $\calD_1(\bfn;0)\le \rho(\rho+1)\ldots k\le k!$. In view of (\ref{3.3}), the conclusion of the lemma therefore follows from (\ref{3.18}).
\end{proof}

Our second bound for $B_{a,b}^{r,h}(p)$ addresses the scenario in which $h=k-r+1$ and $r<k$. This situation amounts to one in which we aim to lift solutions to an intermediate modulus $p^{\rho b}$, and discard any congruences of modulus smaller than $p^{\rho b}$. Again, we provide two variants of this lemma, one with $a\ge 1$ and a second with $a=0$.

\begin{lemma}\label{lemma3.5} Suppose that $a$ and $b$ are natural numbers with $b\ge (r-1)a$. Then $B_{a,b}^{r,\rho}(p)\le k!p^{(r-1)a}$.
\end{lemma}

\begin{proof} Consider fixed natural numbers $a$ and $b$ with $b\ge (r-1)a$, a fixed $r$-tuple $\bfsig\in \Sig_r$, and fixed integers $\xi$ and $\eta$ with $1\le \xi\le p^a$, $1\le \eta\le p^b$ and $\eta\not\equiv \xi\mmod{p}$. In addition, define the integer $\mu_j$ for $\rho\le j\le k$ by putting
$$\mu_j=\begin{cases} 0,&\text{when $\rho+1\le j\le k$,}\\
r-1,&\text{when $j=\rho$.}\end{cases}$$
We denote by $\calD_1(\bfn)$ the set of $\calR(\rho b)$-equivalence classes of solutions of the system of congruences
\begin{equation}\label{3.20}
\sum_{i=1}^r\sig_i(z_i-\eta)^j\equiv n_j\mmod{p^{jb+\mu_ja}}\quad (\rho\le j\le k),
\end{equation}
with $1\le \bfz\le p^{\rho b}$ and $\bfz\equiv \bfxi \mmod{p^{a+1}}$ for some $\bfxi \in \Xi_a^r(\xi)$. Then it follows from (\ref{3.1}) that
\begin{align}
\text{card}(\calC_{a,b}^{\bfsig ,r,\rho}(\bfm;\xi,\eta))&\le \sum_{\substack{1\le n\le p^{\rho b+(r-1)a}\\ n\equiv m_{\rho}\mmod{p^{\rho b}}}}\text{card}(\calD_1(n,m_{\rho+1},\ldots ,m_k))\notag \\
&\le p^{(r-1)a}\max_{1\le \bfn\le p^{kb}}\text{card}(\calD_1(\bfn)).\label{3.21}
\end{align}

\par Following the pattern of the proof of Lemma \ref{lemma3.3}, we next rewrite each variable $z_i$ in the shape $z_i=p^ay_i+\xi$. The hypothesis that $\bfz\equiv \bfxi\mmod{p^{a+1}}$ for some $\bfxi\in \Xi_a^r(\xi)$ again implies that the $r$-tuple $\bfy$ satisfies (\ref{3.9}). Let $\zet=\xi-\eta$ and write $\zet^{-1}$ for the multiplicative inverse of $\zet$ modulo $p^{kb}$. Then we deduce from (\ref{3.20}) that $\text{card}(\calD_1(\bfn))$ is bounded above by the number of $\calR(\rho b-a)$-equivalence classes of solutions of the system of congruences
\begin{equation}\label{3.22}
\sum_{i=1}^r\sig_i(p^ay_i\zet^{-1}+1)^j\equiv n_j(\zet^{-1})^j\mmod{p^{\rho b+(r-1)a}}\quad (\rho\le j\le k),
\end{equation}
with $1\le \bfy\le p^{\rho b-a}$ satisfying (\ref{3.9}). Here, we have made use of the fact that since $b\ge (r-1)a$, then for $j\ge \rho+1$ the validity of a congruence modulo $p^{jb}$ implies that of the corresponding congruence modulo $p^{\rho b+(r-1)a}$.\par

Let $\bfy=\bfw$ be any solution of the system (\ref{3.22}), if such a solution exists. Then we find that all other solutions $\bfy$ satisfy the system of congruences
\begin{equation}\label{3.23}
\sum_{i=1}^r\sig_i\left( (p^ay_i\zet^{-1}+1)^j-(p^aw_i\zet^{-1}+1)^j\right)\equiv 0\mmod{p^{\rho b+(r-1)a}}\quad (\rho\le j\le k).
\end{equation}
Recall the definition (\ref{3.13}) of the polynomials $\psi_j(z)$. Then by taking linear combinations of these congruences, we find as in the proof of Lemma \ref{lemma3.3} that there exist integers $d_{jm}$ $(j-\rho+2\le m\le j)$, for $\rho\le j\le k$, with the property that any solution of (\ref{3.23}) satisfies the system of congruences
\begin{equation}\label{3.24}
(\zet^{-1}p^a)^{j-\rho+1}\sum_{i=1}^r\sig_i\left(\psi_j(y_i)-\psi_j(w_i)\right)\equiv 0\mmod{p^{\rho b+(r-1)a}}\quad (\rho\le j\le k).
\end{equation}

\par Denote by $\calD_2(\bfu)$ the set of $\calR(\rho b-a)$-equivalence classes of solutions of the system of congruences
\begin{equation}\label{3.25}
\sum_{i=1}^r\sig_i\psi_j(y_i)\equiv u_j\mmod{p^{\rho b-a}}\quad (\rho\le j\le k),
\end{equation}
with $1\le \bfy\le p^{\rho b-a}$ satisfying (\ref{3.9}). Note that when $\rho\le j\le k$, one has
$$j-\rho+1\le k-(k-r+1)+1=r.$$
Then it follows from (\ref{3.24}) that
\begin{equation}\label{3.26}
\text{card}(\calD_1(\bfn))\le \max_{1\le \bfu\le p^{kb}}\text{card}(\calD_2(\bfu)).
\end{equation}
With the Jacobian determinant $J(\bfpsi;\bfx)$ defined as in (\ref{3.17}), we find as in the proof of Lemma \ref{lemma3.3} that the solutions $\bfy$ of (\ref{3.25}) counted by $\calD_2(\bfu)$ satisfy $(J(\bfpsi;\bfy),p)=1$. We therefore deduce from Lemma \ref{lemma3.1} that $\text{card}(\calD_2(\bfu))\le \rho(\rho+1)\ldots k\le k!$. In view of (\ref{3.2}), the conclusion of the lemma now follows from (\ref{3.21}) and (\ref{3.26}).
\end{proof}

Again, a variant of Lemma \ref{lemma3.5} supplies an analogue applicable in the special case $a=0$.

\begin{lemma}\label{lemma3.6}
Suppose that $b$ is an integer with $b\ge 1$. Then $B_{0,b}^{r,\rho}(p)\le k!$.
\end{lemma}

\begin{proof} Consider a fixed integer $b$ with $b\ge 1$, a fixed $r$-tuple $\bfsig\in \Sig_r$, and a fixed integer $\eta$ with $1\le \eta\le p^b$. We denote by $\calD_1(\bfn;\eta)$ the set of $\calR (\rho b)$-equivalence classes of solutions of the system of congruences
$$\sum_{i=1}^r\sig_i(z_i-\eta)^j\equiv n_j\mmod{p^{\rho b}}\quad (\rho\le j\le k),$$
with $1\le \bfz\le p^{\rho b}$ and $\bfz\equiv \bfxi\mmod{p}$ for some $\bfxi\in \Xi_0^r(0)$, and for which in addition $\bfz\not\equiv \eta\mmod{p}$. Then it follows from (\ref{3.1}) that
\begin{equation}\label{3.27}
\mathrm{card}(\calC_{0,b}^{\bfsig,r,\rho}(\bfm;0,\eta))\le \max_{1\le \bfn\le p^{\rho b}}\mathrm{card}(\calD_1(\bfn;\eta)).
\end{equation}

\par Recall the definition of the Jacobian determinant $J(\bfy)$ from (\ref{3.19}). Then following the argument concluding the proof of Lemma \ref{lemma3.4}, one discerns that $\calD_1(\bfn;\eta)=\calD_1(\bfn;0)$, and that $\calD_1(\bfn;0)$ counts the solutions of the system of congruences
$$\sum_{i=1}^r\sig_iy_i^j\equiv n_j\mmod{p^{\rho b}}\quad (\rho\le j\le k),$$
with $1\le \bfy\le p^{\rho b}$ satisfying $p\nmid J(\bfy)$. By wielding Lemma \ref{lemma3.1}, we therefore deduce that $\calD_1(\bfn;0)\le \rho(\rho+1)\ldots k\le k!$. In view of (\ref{3.3}), the conclusion of the lemma therefore follows from (\ref{3.27}).
\end{proof}

\section{The conditioning process} As in the analogous treatment of \cite[\S5]{Woo2011a}, the mean value $I_{a,b}^\bfsig(X;\xi,\eta)$ is not, by itself, suitable for use in a repeated efficient congruencing iteration. In this section we show how, without serious loss, one may replace the factor $\grf_b(\bfalp;\eta)^{2s}$ occurring in (\ref{2.12}) by the conditioned factor $\grF_b^\bftau (\bfalp;\eta)^{2u}$ in (\ref{2.13}). Our argument follows very closely the proof of \cite[Lemma 5.1]{Woo2011a}, and so we may be concise by analogy at several points in our discussion.

\begin{lemma}\label{lemma4.1} Let $a$ and $b$ be integers with $b>a\ge 1$. Then one has
$$I_{a,b}(X)\ll K_{a,b}(X)+M^{r-1}I_{a,b+1}(X).$$
\end{lemma}

\begin{proof} Consider fixed integers $\xi$ and $\eta$ with $1\le \xi\le p^a$ and $1\le \eta\le p^b$ with $\eta\not\equiv \xi\mmod{p}$, and an $r$-tuple $\bfsig\in \Sig_r$. Then on considering the underlying Diophantine system, it follows from (\ref{2.12}) that $I_{a,b}^\bfsig (X;\xi,\eta)$ counts the number of integral solutions of the system
\begin{equation}\label{4.1}
\sum_{i=1}^r\sig_i(x_i^j-y_i^j)=\sum_{l=1}^s(v_l^j-w_l^j)\quad (1\le j\le k),
\end{equation}
with
$$1\le \bfx,\bfy,\bfv,\bfw\le X,\quad \bfv\equiv\bfw\equiv \eta\mmod{p^b},$$
and satisfying the property that there exist $\bfxi,\bfzet\in \Xi_a^r(\xi)$ for which
$$\bfx\equiv \bfxi\mmod{p^{a+1}}\quad \text{and}\quad \bfy\equiv \bfzet\mmod{p^{a+1}}.$$
Let $T_1$ denote the number of integral solutions $\bfx$, $\bfy$, $\bfv$, $\bfw$ of the system (\ref{4.1}), counted by $I_{a,b}^\bfsig(X;\xi,\eta)$, in which the $2s$ integers $v_1,\ldots,v_s$ and $w_1,\ldots ,w_s$ together lie in at most $r-1$ distinct residue classes modulo $p^{b+1}$, and let $T_2$ denote the corresponding number of solutions in which these integers together occupy at least $r$ distinct residue classes modulo $p^{b+1}$. Then
$$I_{a,b}^\bfsig (X;\xi,\eta)\le T_1+T_2.$$

\par The argument of the proof of \cite[Lemma 5.1]{Woo2011a} leading to equation (5.2) of that paper shows, mutatis mutandis, that
\begin{align*}
T_1&\ll \sum_{\substack{1\le \eta_1,\ldots ,\eta_{r-1}\le p^{b+1}\\ \bfeta \equiv \eta\mmod{p^b}}}\sum_{i=1}^r\oint |\grF_a^\bfsig (\bfalp;\xi)^2\grf_{b+1}(\bfalp;\eta_i)^{2s}|\d\bfalp \\
&\ll p^{r-1}\max_{\substack{1\le \eta_0\le p^{b+1}\\ \eta_0\not\equiv\xi\mmod{p}}}I_{a,b+1}^\bfsig (X;\xi,\eta_0).
\end{align*}
On the other hand, the argument of the proof of \cite[Lemma 5.1]{Woo2011a} leading to equation (5.3) of that paper shows, mutatis mutandis, that for some $\bftau \in \Sig_r$ one has
\begin{align*}
T_2&\ll \Bigl( \oint |\grF_a^\bfsig (\bfalp;\xi)^2\grF_b^\bftau (\bfalp;\eta)^{2u}|\d\bfalp \Bigr)^{1/(2u)}\Bigl( \oint |\grF_a^\bfsig (\bfalp;\xi)^2\grf_b(\bfalp;\eta)^{2s}|\d\bfalp \Bigr)^{1-1/(2u)}\\
&\ll \left( K_{a,b}^{\bfsig,\bftau}(X;\xi,\eta)\right)^{1/(2u)}\left( I_{a,b}^\bfsig (X;\xi,\eta)\right)^{1-1/(2u)}.
\end{align*}
Thus we deduce from (\ref{2.14}) and (\ref{2.15}) that
$$I_{a,b}(X)\ll M^{r-1}I_{a,b+1}(X)+(K_{a,b}(X))^{1/(2u)}(I_{a,b}(X))^{1-1/(2u)}.$$
The conclusion of the lemma follows immediately.
\end{proof}

We next obtain an estimate that enables us to truncate the conditioning process. Here we recall that the exponent $\kap_{s+r}$ is a positive number, with
$$\kap_{s+r}\le \max\{s+r,\tfrac{1}{2}k(k+1)\},$$
which measures the strength of the permissible exponent $\lam_s^*$ by means of the relation (\ref{2.2}). We have in mind the choices for $\kap_{s+r}$ presented in equations (\ref{2.5}) and (\ref{2.6}). Finally, it is convenient to write $\kap$ for $\kap_{s+r}$, since confusion is easily avoided.

\begin{lemma}\label{lemma4.2} Let $a$, $b$ and $H$ be positive integers with
$$0<2(b-a)\le H\le \tet^{-1}-b.$$
Then provided that $s\ge 3r$, one has
$$M^{H(r-1)}I_{a,b+H}(X)\ll M^{-(r+2)H/2}X^\del (X/M^b)^{2s}(X/M^a)^{2r-\kap+\eta_{s+r}}.$$
\end{lemma}

\begin{proof} On considering the underlying Diophantine equations, we find from (\ref{2.12}) that when $1\le \xi\le p^a$, $1\le \eta\le p^{b+H}$ and $\bfsig\in \Sig_r$, one has
$$I_{a,b+H}^\bfsig(X;\xi,\eta)\le \oint|\grf_a(\bfalp;\xi)^{2r}\grf_{b+H}(\bfalp;\eta)^{2s}|\d\bfalp .$$
Applying H\"older's inequality together with Lemma \ref{lemma2.1}, therefore, we obtain
\begin{align*}
I_{a,b+H}^\bfsig (X;\xi,\eta)&\le \Bigl( \oint|\grf_a(\bfalp;\xi)|^{2s+2r}\d\bfalp\Bigr)^{r/(s+r)}\Bigl( \oint |\grf_{b+H}(\bfalp;\eta)|^{2s+2r}\d\bfalp \Bigr)^{s/(s+r)}\\
&\ll (J_{s+r}(X/M^a))^{r/(s+r)}(J_{s+r}(X/M^{b+H}))^{s/(s+r)}.
\end{align*}
We thus deduce from (\ref{2.2}) and (\ref{2.8}) that
\begin{align*}
I_{a,b+H}(X)&\ll \left( (X/M^a)^{r/(s+r)}(X/M^{b+H})^{s/(s+r)}\right)^{2s+2r-\kap+\eta_{s+r}+\del}\\
&\ll X^\del (X/M^a)^{2r-\kap+\eta_{s+r}}(X/M^b)^{2s}\Ups ,
\end{align*}
where
$$\Ups=(M^{b-a+H})^{\kap s/(s+r)}M^{-2sH}.$$

\par We may suppose that $s+r\ge \kap$, $H\ge 2(b-a)$ and $s\ge 3r$, and hence
\begin{align*}
rH+(b-a+H)\kap s/(s+r)-2sH&\le rH+\tfrac{3}{2}\kap sH/(s+r)-2sH\\
&\le \tfrac{1}{2}(2r-s)H\le -\tfrac{1}{2}rH.
\end{align*}
Consequently, one has
$$M^{H(r-1)}\Ups\le M^{-(r+2)H/2},$$
whence
$$M^{H(r-1)}I_{a,b+H}(X)\ll M^{-(r+2)H/2}X^\del (X/M^a)^{2r-\kap +\eta_{s+r}}(X/M^b)^{2s},$$
and the conclusion of the lemma follows.
\end{proof}

The repeated application of Lemma \ref{lemma4.1} in combination with Lemma \ref{lemma4.2} yields the conditioning lemma underpinning the efficient congruencing process.

\begin{lemma}\label{lemma4.3} Let $a$ and $b$ be integers with $1\le a<b$, and put $H=2(b-a)$. Suppose that $b+H\le \tet^{-1}$ and $s\ge 3r$. Then there exists an integer $h$ with $0\le h<H$ having the property that
$$I_{a,b}(X)\ll M^{h(r-1)}K_{a,b+h}(X)+M^{-(r+2)H/2}X^\del (X/M^b)^{2s}(X/M^a)^{2r-\kap+\eta_{s+r}}.$$
\end{lemma}

\begin{proof} Repeated application of Lemma \ref{lemma4.1} shows that whenever $a$ and $b$ are positive integers with $b>a\ge 1$, and $H=2(b-a)$, then
\begin{equation}\label{4.2}
I_{a,b}(X)\ll \sum_{h=0}^{H-1}M^{h(r-1)}K_{a,b+h}(X)+M^{H(r-1)}I_{a,b+H}(X).
\end{equation}
The desired conclusion therefore follows on applying Lemma \ref{lemma4.2} to estimate the second term on the right hand side of (\ref{4.2}).
\end{proof}

\section{The efficient congruencing step, I} Our goal in this section is to convert latent congruence information within the mean value $K_{a,b}(X)$ into a form useful in subsequent iterations, and this we achieve using the work of \S3. The two basic approaches of \S3 yield two different manifestations of the efficient congruencing step, and these we examine in separate lemmata.

\begin{lemma}\label{lemma5.1} Suppose that $a$ and $b$ are integers with $1\le a<b\le \tet^{-1}$. Then one has
$$K_{a,b}(X)\ll M^{\frac{1}{2}r(r-1)(b+a)}(M^{kb-a})^r(J_{s+r}(X/M^b))^{1-r/s}(I_{b,kb}(X))^{r/s}.$$
\end{lemma}

\begin{proof} Consider fixed integers $\xi$ and $\eta$ with $1\le \xi\le p^a$, $1\le \eta\le p^b$ and $\eta\not\equiv \xi\mmod{p}$, and $r$-tuples $\bfsig, \bftau\in \Sig_r$. Then by orthogonality, the mean value $K_{a,b}^{\bfsig,\bftau}(X;\xi,\eta)$ defined in (\ref{2.13}) counts the number of integral solutions of the system
\begin{equation}\label{5.1}
\sum_{i=1}^r\sig_i(x_i^j-y_i^j)=\sum_{l=1}^u\sum_{m=1}^r\tau_m(v_{lm}^j-w_{lm}^j)\quad (1\le j\le k),
\end{equation}
in which, for some $\bfxi,\bfzet\in \Xi^r_a(\xi)$, one has
$$1\le \bfx,\bfy\le X,\quad \bfx\equiv \bfxi\mmod{p^{a+1}}\quad \text{and}\quad \bfy\equiv \bfzet\mmod{p^{a+1}},$$
and for $1\le l\le u$, for some $\bfeta_l,\bfnu_l\in \Xi^r_b(\eta)$, one has
$$1\le \bfv_l,\bfw_l\le X,\quad \bfv_l\equiv \bfeta_l\mmod{p^{b+1}}\quad \text{and}\quad \bfw_l\equiv \bfnu_l\mmod{p^{b+1}}.$$
As in the argument of the proof of \cite[Lemma 6.1]{Woo2011a}, an application of the Binomial Theorem shows that these solutions satisfy the system of congruences
\begin{equation}\label{5.2}
\sum_{i=1}^r\sig_i(x_i-\eta)^j\equiv \sum_{i=1}^r\sig_i(y_i-\eta)^j\mmod{p^{jb}}\quad (1\le j\le k).
\end{equation}

\par We now make use of the work of \S3, writing
$$\grG_{a,b}^\bfsig (\bfalp;\xi,\eta;\bfm)=\sum_{\bfzet \in \calC_{a,b}^{\bfsig,r,k}(\bfm;\xi,\eta)}\prod_{i=1}^r\grf_{kb}(\sig_i\bfalp;\zet_i).$$
Then on considering the underlying Diophantine system, we see from (\ref{5.1}) and (\ref{5.2}) that
\begin{equation}\label{5.3}
K_{a,b}^{\bfsig,\bftau}(X;\xi,\eta)=\sum_{m_1=1}^{p^b}\ldots \sum_{m_k=1}^{p^{kb}}\oint |\grG_{a,b}^\bfsig (\bfalp;\xi,\eta;\bfm)^2\grF_b^\bftau(\bfalp;\eta)^{2u}|\d\bfalp .
\end{equation}
An application of Cauchy's inequality leads via Lemma \ref{lemma3.3} to the bound
\begin{align*}
|\grG_{a,b}^\bfsig (\bfalp;\xi,\eta;\bfm)|^2&\le \text{card}(\calC_{a,b}^{\bfsig,r,k}(\bfm;\xi,\eta))\sum_{\bfzet\in \calC_{a,b}^{\bfsig,r,k}(\bfm;\xi,\eta)}\prod_{i=1}^r|\grf_{kb}(\bfalp;\zet_i)|^2\\
&\ll M^{\frac{1}{2}r(r-1)(a+b)}\sum_{\bfzet\in \calC_{a,b}^{\bfsig,r,k}(\bfm;\xi,\eta)}\prod_{i=1}^r|\grf_{kb}(\bfalp;\zet_i)|^2,
\end{align*}
whence
$$K_{a,b}^{\bfsig,\bftau}(X;\xi,\eta)\ll M^{\frac{1}{2}r(r-1)(a+b)}\sum_{\substack{1\le \bfzet\le p^{kb}\\ \bfzet\equiv \xi\mmod{p^a}}}\oint \Bigl( \prod_{i=1}^r|\grf_{kb}(\bfalp;\zet_i)|^2\Bigr) |\grF_b^\bftau (\bfalp;\eta)|^{2u}\d\bfalp .$$
As in the argument of the proof of \cite[Lemma 6.1]{Woo2011a} leading to equation (6.7) of the latter paper, from here an application of H\"older's inequality yields the upper bound
\begin{align}
K_{a,b}^{\bfsig,\bftau}&(X;\xi,\eta)\notag \\
&\ll M^{\frac{1}{2}r(r-1)(a+b)}(M^{kb-a})^r\max_{\substack{1\le \zet \le p^{kb}\\ \zet\equiv \xi\mmod{p^a}}}\oint |\grf_{kb}(\bfalp;\zet)^{2r}\grF_b^\bftau(\bfalp;\eta)^{2u}|\d\bfalp .\label{5.4}
\end{align}

\par Next we apply H\"older's inequality to the integral on the right hand side of (\ref{5.4}) to obtain
$$\oint |\grf_{kb}(\bfalp;\zet)^{2r}\grF_b^\bftau (\bfalp;\eta)^{2u}|\d\bfalp \le U_1^{1-r/s}U_2^{r/s},$$
where
$$U_1=\oint |\grF_b^\bftau(\bfalp;\eta)|^{2u+2}\d\bfalp \le \oint |\grf_b(\bfalp;\eta)|^{2s+2r}\d\bfalp$$
and
$$U_2=I_{b,kb}^\bftau(X;\eta,\zet).$$
Notice here that since $\eta\not\equiv \xi\mmod{p}$ and $\zet\equiv \xi\mmod{p^a}$ with $a\ge 1$, we have $\zet\not\equiv \eta\mmod{p}$. In this way we deduce from Lemma \ref{lemma2.1} that
$$\oint|\grf_{kb}(\bfalp;\zet)^{2r}\grF_b^\bftau(\bfalp;\eta)^{2u}|\d\bfalp \ll (J_{s+r}(X/M^b))^{1-r/s}(I_{b,kb}(X))^{r/s},$$
and the conclusion of the lemma follows from (\ref{5.4}).
\end{proof}

A variant of the argument employed to establish Lemma \ref{lemma5.1} makes use of Lemma \ref{lemma3.5} in place of Lemma \ref{lemma3.3}.

\begin{lemma}\label{lemma5.2}
Suppose that $a$ and $b$ are integers with $1\le a<b\le \tet^{-1}$ and $b\ge (r-1)a$. Then one has
$$K_{a,b}(X)\ll M^{(r-1)a}(M^{\rho b-a})^r(J_{s+r}(X/M^b))^{1-r/s}(I_{b,\rho b}(X))^{r/s}.$$
\end{lemma}

\begin{proof} Initially we follow the argument of the proof of Lemma \ref{lemma5.1}, identifying $K_{a,b}^{\bfsig,\bftau}(X;\xi,\eta)$ with the number of integral solutions of the system (\ref{5.1}) with its attendant conditions, and observing that the system of congruences (\ref{5.2}) necessarily holds for each solution. We now write
$$\grG_{a,b}^\bfsig (\bfalp;\xi,\eta;\bfm)=\sum_{\bfzet \in \calC_{a,b}^{\bfsig,r,\rho}(\bfm;\xi,\eta)}\prod_{i=1}^r\grf_{\rho b}(\sig_i\bfalp;\zet_i),$$
and note as before that the relation (\ref{5.3}) again holds. An application of Cauchy's inequality in this instance leads from Lemma \ref{lemma3.5} to the estimate
\begin{align*}
|\grG_{a,b}^\bfsig (\bfalp;\xi,\eta;\bfm)|^2&\le \text{card}(\calC_{a,b}^{\bfsig,r,\rho}(\bfm;\xi,\eta))\sum_{\bfzet\in \calC_{a,b}^{\bfsig,r,\rho}(\bfm;\xi,\eta)}\prod_{i=1}^r|\grf_{\rho b}(\bfalp;\zet_i)|^2\\
&\ll M^{(r-1)a}\sum_{\bfzet\in \calC_{a,b}^{\bfsig,r,\rho}(\bfm;\xi,\eta)}\prod_{i=1}^r|\grf_{\rho b}(\bfalp;\zet_i)|^2,
\end{align*}
and hence
$$K_{a,b}^{\bfsig,\bftau}(X;\xi,\eta)\ll M^{(r-1)a}\sum_{\substack{1\le \bfzet\le p^{\rho b}\\ \bfzet\equiv \xi\mmod{p^a}}}\oint \Bigl( \prod_{i=1}^r|\grf_{\rho b}(\bfalp;\zet_i)|^2\Bigr) |\grF_b^\bftau (\bfalp;\eta)|^{2u}\d\bfalp .$$
From here, as in the argument leading to (\ref{5.4}) above, one obtains
\begin{align}
K_{a,b}^{\bfsig,\bftau}&(X;\xi,\eta)\notag \\
&\ll M^{(r-1)a}(M^{\rho b-a})^r\max_{\substack{1\le \zet \le p^{\rho b}\\ \zet\equiv \xi\mmod{p^a}}}\oint |\grf_{\rho b}(\bfalp;\zet)^{2r}\grF_b^\bftau(\bfalp;\eta)^{2u}|\d\bfalp .\label{5.5}
\end{align}

\par Applying H\"older's inequality as in the concluding paragraph of the proof of Lemma \ref{lemma5.1}, we deduce that
$$\oint|\grf_{\rho b}(\bfalp;\zet)^{2r}\grF_b^\bftau(\bfalp;\eta)^{2u}|\d\bfalp \ll (J_{s+r}(X/M^b))^{1-r/s}(I_{b,\rho b}(X))^{r/s},$$
and the conclusion of the lemma follows from (\ref{5.5}).
\end{proof}

A crude but simple upper bound for $K_{a,b}(X)$ is useful in simplifying the argument to come.

\begin{lemma}\label{lemma5.3}
Suppose that $a$ and $b$ are integers with $0\le a<b\le \tet^{-1}$. Then
$$\llbracket K_{a,b}(X)\rrbracket \ll X^{\eta_{s+r}+\del}(M^{b-a})^\kap.$$
\end{lemma}

\begin{proof} We adapt the argument of the proof of \cite[Lemma 6.2]{Woo2011a}. Consider fixed integers $\xi$ and $\eta$ with $1\le \xi\le p^a$ and $1\le \eta\le p^b$, and $r$-tuples $\bfsig,\bftau\in \Sig_r$. Then from (\ref{2.13}) it follows by orthogonality combined with H\"older's inequality that
\begin{align*}
K_{a,b}^{\bfsig,\bftau}(X;\xi,\eta)&\le \oint |\grf_a(\bfalp;\xi)^{2r}\grf_b(\bfalp;\eta)^{2s}|\d\bfalp \\
&\le \Bigl( \oint |\grf_a(\bfalp;\xi)|^{2s+2r}\d\bfalp \Bigr)^{r/(s+r)}\Bigl( \oint |\grf_b(\bfalp;\eta)|^{2s+2r}\d\bfalp\Bigr)^{s/(s+r)}.
\end{align*}
Consequently, Lemma \ref{lemma2.1} delivers the bound
$$K_{a,b}(X)\ll (J_{s+r}(X/M^a))^{r/(s+r)}(J_{s+r}(X/M^b))^{s/(s+r)},$$
whence
\begin{align*}
\llbracket K_{a,b}(X)\rrbracket&\ll \frac{X^\del \left( (X/M^a)^{r/(s+r)}(X/M^b)^{s/(s+r)}\right)^{2s+2r-\kap+\eta_{s+r}}}{(X/M^b)^{2s}(X/M^a)^{2r-\kap}}\\
&\ll X^{\eta_{s+r}+\del}(M^{b-a})^{\kap s/(s+r)}\ll X^{\eta_{s+r}+\del}(M^{b-a})^\kap .
\end{align*}
This completes the proof of the lemma.
\end{proof}

\section{The pre-congruencing step} In order to fix choices for $\xi$ and $\eta$ in \S\S3--5, one must first initiate the congruencing process. It is here that the choice for the prime number $p$ is fixed once and for all. Before delving further into the details of this pre-congruencing step, we pause to introduce some additional notation. We amend the definition of the set $\Xi_c^r(\xi)$ from the discussion leading to (\ref{2.11}) as follows. When $\Eta\subseteq \{1,\ldots ,p\}$, we denote by $\Xi(\Eta)$ the set of $r$-tuples $(\xi_1,\ldots ,\xi_r)$ satisfying $1\le \bfxi\le p$ and in addition the property that one has neither $\xi_i\equiv \zet\mmod{p}$ for any $\zet\in \Eta$ $(1\le i\le r)$, nor $\xi_i\equiv \xi_j\mmod{p}$ for any $i$ and $j$ with $1\le i<j\le r$. Recalling (\ref{2.10}), we next define the exponential sum $\grF(\bfalp;\Eta)$ by putting
\begin{equation}\label{6.1}
\grF(\bfalp;\Eta)=\sum_{\bfxi\in \Xi(\Eta)}\prod_{i=1}^r\grf_1(\bfalp;\xi_i).
\end{equation}
Also, when $\Eta$ is a subset of $\{1,\ldots,p\}$ with cardinality $k-r$, we write
$$L(\bfalp;\Eta)=\prod_{\zet\in \Eta}\grf_1(\bfalp;\zet).$$
Finally, we write
\begin{align}
\Itil_c(X;\eta)&=\oint |\grF(\bfalp;\{\eta\})^2\grf_c(\bfalp;\eta)^{2s}|\d\bfalp ,\label{6.2}\\
\Ktil_c^\bftau(X;\eta)&=\oint |\grF(\bfalp;\{\eta\})^2\grF_c^\bftau (\bfalp;\eta)^{2u}|\d\bfalp ,\label{6.3}\\
\Ktil_c(X)&=\max_{1\le \eta\le p^c}\max_{\bftau\in\Sig_r}K_c^\bftau (X;\eta).\label{6.4}
\end{align}

\begin{lemma}\label{lemma6.1}
Suppose that $s\ge \max\{k+1-r,2r\}$ and $\kap\le s+r$. Then there exists a prime number $p$ with $M<p\le 2M$, and an integer $h\in\{0,1\}$, for which one has
$$J_{s+r}(X)\ll M^{2s+h(r-1)}\Ktil_{1+h}(X).$$
\end{lemma}

\begin{proof} We adapt the argument of the proof of \cite[Lemma 3.2]{Woo2011a}. The quantity $J_{s+r}(X)$ counts the number of integral solutions of the system
$$\sum_{i=1}^{s+r}(x_i^j-y_i^j)=0\quad (1\le j\le k),$$
with $1\le \bfx,\bfy\le X$. Let $T_0$ denote the number of such solutions in which $x_i=x_m$ for some $i$ and $m$ with $1\le i<m\le k$, and let $T_1$ denote the corresponding number of solutions with $x_i=x_m$ for no $i$ and $m$ with $1\le i<m\le k$. Then $J_{s+r}(X)=T_0+T_1$.\par

Write $\Xi^*(p)$ for the set of $k$-tuples $\bfxi=(\xi_1,\ldots ,\xi_k)$, with $1\le \bfxi\le p$, and satisfying the property that $\xi_i\equiv \xi_m\pmod{p}$ for no $i$ and $m$ with $1\le i<m\le k$. In addition, define the exponential sum $\grF^*(\bfalp)$ by putting
$$\grF^*(\bfalp)=\sum_{\bfxi\in \Xi^*(p)}\prod_{i=1}^k\grf_1(\bfalp;\xi_i),$$
and write
\begin{equation}\label{6.5}
I^*(X)=\oint |\grF^*(\bfalp)^2\grf_0(\bfalp;0)^{2s+2r-2k}|\d\bfalp .
\end{equation}
Then the argument of the proof of \cite[Lemma 3.2]{Woo2011a} leading to equations (3.14) and (3.15) of the latter paper reveals that a prime number $p$ exists, with $M<p\le 2M$, for which
$$T_0\ll (J_{s+r}(X))^{1-1/(2s+2r)}\quad \text{and}\quad T_1\ll (I^*(X))^{1/2}(J_{s+r}(X))^{1/2}.$$
We thus infer that
\begin{equation}\label{6.6}
J_{s+r}(X)\ll 1+I^*(X)\ll I^*(X).
\end{equation}

\par Next, splitting the summation in the definition (\ref{2.10}) of $\grf_0(\bfalp;0)$ into arithmetic progressions modulo $p$ and applying H\"older's inequality, we obtain
$$|\grf_0(\bfalp;0)|^{2s+2r-2k}\le p^{2s+2r-2k-1}\sum_{\eta=1}^p|\grf_1(\bfalp;\eta)|^{2s+2r-2k}.$$
It therefore follows from (\ref{6.5}) that
\begin{equation}\label{6.7}
I^*(X)\ll M^{2s+2r-2k}\max_{1\le \eta\le p}T_2(\eta),
\end{equation}
where
$$T_2(\eta)=\oint |\grF^*(\bfalp)^2\grf_1(\bfalp;\eta)^{2s+2r-2k}|\d\bfalp .$$
By orthogonality, the mean value $T_2(\eta)$ counts the integral solutions of the system
$$\sum_{i=1}^k(x_i^j-y_i^j)=\sum_{l=1}^{s+r-k}(v_l^j-w_l^j)\quad (1\le j\le k),$$
with
$$1\le \bfx,\bfy,\bfv,\bfw\le X,\quad \bfv\equiv\bfw\equiv \eta\mmod{p},$$
and satisfying the property that there exist $\bfxi,\bfzet \in \Xi^*(p)$ for which
$$\bfx\equiv\bfxi\mmod{p}\quad \text{and}\quad \bfy\equiv \bfzet\mmod{p}.$$

\par Consider a fixed choice of $\bfxi\in \Xi^*(p)$. One has $\xi_l\equiv \eta\pmod{p}$ for at most one index $l$ with $1\le l\le k$. Since we suppose that $1\le r\le k-1$, it follows that one may relabel indices in such a way that $(\xi_1,\ldots ,\xi_r)\in \Xi_0^r(0)$ and $\xi_i\equiv \eta\mmod{p}$ for no index $i$ with $1\le i\le r$. One may do likewise with the variables $\bfy$. Notice that when $(\xi_1,\ldots ,\xi_k)\in \Xi^*(p)$ and $(\xi_1,\ldots ,\xi_r)\in \Xi^r_0(0)$, then necessarily $\xi_i\equiv \xi_j\mmod{p}$ for no indices $i$ and $j$ with $1\le i\le r$ and $r+1\le j\le k$. On considering the underlying Diophantine equations, therefore, we find that
$$T_2(\eta)\ll \oint \Bigl|\sum_{\substack{\Eta\subseteq \{1,\ldots ,p\}\\ \mathrm{card}(\Eta)=k-r}}\grF(\bfalp;\Eta\cup\{\eta\})L(\bfalp;\Eta)\Bigr|^2|\grf_1(\bfalp;\eta)|^{2s+2r-2k}\d\bfalp .$$

\par An application of the elementary inequality
\begin{equation}\label{6.8}
|z_1\ldots z_n|\le |z_1|^n+\ldots +|z_n|^n
\end{equation}
reveals that
\begin{align*}
L(\bfalp;\Eta)\ll \sum_{\zet\in \Eta}|\grf_1(\bfalp;\zet)|^{k-r},
\end{align*}
and thus we deduce via Cauchy's inequality that
$$T_2(\eta)\ll p^{k-r}\sum_{\substack{\Eta\subseteq \{1,\ldots ,p\}\\ \mathrm{card}(\Eta)=k-r}}\sum_{\zet\in \Eta}\oint |\grF(\bfalp;\Eta\cup\{\eta\})\grf_1(\bfalp;\zet)^{k-r}\grf_1(\bfalp;\eta)^{s+r-k}|^2\d\bfalp .$$
A second application of (\ref{6.8}) shows that
$$|\grf_1(\bfalp;\zet)^{k-r}\grf_1(\bfalp;\eta)^{s+r-k}|^2\ll |\grf_1(\alp;\zet)|^{2s}+|\grf_1(\alp;\eta)|^{2s},$$
and hence we conclude that
$$T_2(\eta)\ll p^{2k-2r}\max_{\substack{\Eta\subseteq \{1,\ldots ,p\}\\ \mathrm{card}(\Eta)=k-r}}\max_{\zet\in \Eta \cup \{\eta\}}\oint |\grF(\bfalp;\Eta\cup \{\eta\})^2\grf_1(\bfalp;\zet)^{2s}|\d\bfalp .$$
Finally, a consideration of the underlying Diophantine system permits the last estimate to be simplified, so that on recalling (\ref{6.2}) we arrive at the bound
\begin{align*}
T_2(\eta)&\ll p^{2k-2r}\max_{1\le \zet\le p}\oint |\grF(\bfalp;\{\zet\})^2\grf_1(\bfalp;\zet)^{2s}|\d\bfalp \\
&\ll M^{2k-2r}\max_{1\le \zet\le p}\Itil_1(X;\zet).
\end{align*}
Returning to (\ref{6.6}) and (\ref{6.7}), we may thus conclude that
\begin{equation}\label{6.9}
J_{s+r}(X)\ll I^*(X)\ll M^{2s}\max_{1\le \eta\le p}\Itil_1(X;\eta).
\end{equation}

\par The mean value $\Itil_1(X;\eta)$ counts the number of integral solutions of the system
$$\sum_{i=1}^r(x_i^j-y_i^j)=\sum_{l=1}^s(v_l^j-w_l^j)\quad (1\le j\le k),$$
with
$$1\le \bfx,\bfy,\bfv,\bfw\le X,\quad \bfv\equiv \bfw\equiv \eta\mmod{p},$$
and satisfying the property that there exist $\bfxi,\bfzet\in \Xi(\{\eta\})$ for which
$$\bfx\equiv \bfxi\mmod{p}\quad \text{and}\quad \bfy\equiv \bfzet\mmod{p}.$$
Let $T_3$ denote the number of such solutions in which the $2s$ integers $v_1,\ldots ,v_s$ and $w_1,\ldots ,w_s$ together occupy at least $r$ distinct residue classes modulo $p^2$, and let $T_4$ denote the corresponding number of solutions in which these integers together lie in at most $r-1$ distinct residue classes modulo $p^2$. Then we see that
\begin{equation}\label{6.10}
\Itil_1(X;\eta)=T_3+T_4.
\end{equation}

\par The argument of the proof of \cite[Lemma 5.1]{Woo2011a} leading to equation (\ref{5.3}) of that paper shows, mutatis mutandis, that for some $\bftau\in \Sig_r$, one has
\begin{align*}
T_3\ll &\, \Bigl( \oint |\grF(\bfalp;\{\eta\})^2\grF_1^\bftau (\bfalp ;\eta)^{2u}|\d\bfalp \Bigr)^{1/(2u)}\\
&\, \times \Bigl( \oint |\grF(\bfalp;\{\eta\})^2\grf_1(\bfalp ;\eta)^{2s}|\d\bfalp \Bigr)^{1-1/(2u)}.
\end{align*}
Then on recalling (\ref{6.2}) and (\ref{6.3}), we deduce from (\ref{6.10}) that
$$\Itil_1(X;\eta)\ll \left( \Ktil_1^\bftau (X;\eta)\right)^{1/(2u)}\left( \Itil_1(X;\eta)\right)^{1-1/(2u)}+T_4,$$
whence
\begin{equation}\label{6.11}
\Itil_1(X;\eta)\ll \Ktil_1^\bftau(X;\eta)+T_4.
\end{equation}

\par On the other hand, the argument of the proof of \cite[Lemma 5.1]{Woo2011a} leading to equation (5.2) of that paper  shows that
$$T_4\ll \sum_{\substack{1\le \eta_1,\ldots,\eta_{r-1}\le p^2\\ \bfeta\equiv \eta\mmod{p}}}\sum_{i=1}^{r-1}\oint |\grF(\bfalp;\{\eta\})^2\grf_2(\bfalp;\eta_i)^{2s}|\d\bfalp .$$
Such a conclusion may also be extracted from the argument of the proof of Lemma \ref{lemma4.1} above. A consideration of the underlying Diophantine system therefore shows that
$$T_4\ll M^{r-1}\max_{1\le \zet\le p^2}\oint |\grF(\bfalp;\{\zet\})^2\grf_2(\bfalp;\zet)^{2s}|\d\bfalp .$$
In view of the relation (\ref{6.11}), therefore, we deduce that
\begin{equation}\label{6.12}
\Itil_1(X;\eta)\ll \Ktil_1^\bftau(X;\eta)+M^{r-1}\max_{1\le \zet\le p^2}\Itil_2(X;\zet).
\end{equation}

We may analyse the mean value $\Itil_2(X;\zet)$ just as in our treatment of $\Itil_1(X;\eta)$ above, and thus we deduce that for some $\bftau\in \Sig_r$, one has
$$\Itil_2(X;\zet)\ll \Ktil_2^\bftau(X;\zet)+M^{r-1}\max_{1\le \tet\le p^3}\Itil_3(X;\tet).$$
On recalling (\ref{6.4}), therefore, we find from (\ref{6.12}) that
$$\Itil_1(X;\eta)\ll \Ktil_1(X)+M^{r-1}\Ktil_2(X)+M^{2r-2}\max_{1\le \eta\le p^3}\Itil_3(X;\eta).$$
Thus we deduce from (\ref{6.9}) that there exists an integer $h\in\{0,1\}$ for which one has
\begin{equation}\label{6.13}
J_{s+r}(X)\ll M^{2s+h(r-1)}\Ktil_{1+h}(X)+M^{2s+2r-2}\max_{1\le \eta\le p^3}\Itil_3(X;\eta).
\end{equation}

Next, on considering the underlying Diophantine system, an application of H\"older's inequality in combination with Lemma \ref{lemma2.1} confirms that
\begin{align*}
\Itil_3(X;\eta)&\ll \Bigl( \oint |\grf_0(\bfalp;0)|^{2s+2r}\d\bfalp \Bigr)^{r/(s+r)}\Bigl( \max_{1\le \eta \le p^3}\oint |\grf_3(\bfalp;\eta)|^{2s+2r}\d\bfalp \Bigr)^{s/(s+r)}\\
&\ll (X^{2s+2r-\kap+\del})^{r/(s+r)}\left( (X/M^3)^{2s+2r-\kap+\del}\right)^{s/(s+r)}.
\end{align*}
Hence we have
$$M^{2s+2r-2}\Itil_3(X;\eta)\ll X^{2s+2r-\kap+\del}M^\ome,$$
where
$$\ome =2s+2r-2-\frac{3s}{s+r}(2s+2r-\kap+\del)\le 2r-2-4s+3s\kap/(s+r).$$
But by hypothesis, one has $\kap\le s+r$ and $s\ge 2r$, and thus
$$\ome\le (2r-2-4s)+3s\le -2.$$
Consequently, we derive the upper bound
$$M^{2s+2r-2}\Itil_3(X;\eta)\ll X^{2s+2r-\kap-2\del}\ll X^{-\del}J_{s+r}(X).$$
On recalling (\ref{6.13}), therefore, we see that there exists an integer $h\in\{0,1\}$ for which
$$J_{s+r}(X)\ll X^{-\del}J_{s+r}(X)+M^{2s+h(r-1)}\Ktil_{1+h}(X).$$
The conclusion of the lemma follows at once.
\end{proof}

We now fix the prime number $p$, once and for all, so that the upper bound for $J_{s+r}(X)$ claimed in the conclusion of Lemma \ref{lemma6.1} holds. In analysing the iterative process, we shall find it useful to have 
available versions of Lemmata \ref{lemma5.1} and \ref{lemma5.2} valid also when $a=0$. It is for this purpose that we prepared Lemma \ref{lemma6.1}, as we now make transparent. In this context, we view the mean 
value $\Ktil_c^\bftau (X;\eta)$ defined in (\ref{6.3}) as a surrogate for $K^{{\mathbf 1},\bftau}_{0,c}(X;0,\eta)$. Here, the implicit condition on variables avoiding the congruence class $\eta$ 
modulo $p$, captured through the exponential sum $\grF(\bfalp;\{\eta\})$ defined in (\ref{6.1}), provides the correct analogue of the condition $\xi\not \equiv \eta\mmod{p}$. With this discussion in mind, we henceforth 
adopt the convention that when $b\in\{1,2\}$, one is to interpret the expression $K_{0,b}(X)$ as $\Ktil_b(X)$.

\begin{lemma}\label{lemma6.2} Suppose that $a$ and $b$ are integers with $0\le a<b\le \tet^{-1}$. Suppose further that $s\ge \max\{k+1-r,2r\}$ and $\kap\le s+r$, and that when $a=0$ one has $b=1$ or $2$. Then
$$K_{a,b}(X)\ll M^{\frac{1}{2}r(r-1)(b+a)}(M^{kb-a})^r(J_{s+r}(X/M^b))^{1-r/s}(I_{b,kb}(X))^{r/s}.$$
\end{lemma}

\begin{proof} When $a\ge 1$ the conclusion asserted by the lemma is an immediate consequence of that supplied by Lemma \ref{lemma5.1}. We therefore focus attention on the situation in which $a=0$ and $b\in\{1,2\}$. Here, we find from (\ref{6.3}) and (\ref{6.4}) that
\begin{equation}\label{6.14}
K_{0,b}(X)=\Ktil_b(X)=\max_{1\le \eta\le p^b}\max_{\bftau\in\Sig_r}\oint |\grF(\bfalp;\{\eta\})^2\grF_b^\bftau (\bfalp;\eta)^{2u}|\d\bfalp .
\end{equation}
Imitating the argument of the proof of Lemma \ref{lemma5.1}, and substituting the application of Lemma \ref{lemma3.4} for our earlier use of Lemma \ref{lemma3.3}, with the discussion of the preamble to the present lemma in mind, we find that
$$K_{0,b}(X)\ll M^{\frac{1}{2}r(r-1)b}(M^{kb})^r(J_{s+r}(X/M^b))^{1-r/s}(I_{b,kb}(X))^{r/s},$$
and thus the desired conclusion does indeed hold when $a=0$ and $b\in\{1,2\}$.
\end{proof}

\begin{lemma}\label{lemma6.3} Suppose that $a$ and $b$ are integers with $0\le a<b\le \tet^{-1}$ and $b\ge (r-1)a$. Suppose further that $s\ge \max\{k+1-r,2r\}$ and $\kap\le s+r$, and that when $a=0$ one has $b=1$ or $2$. Then
$$K_{a,b}(X)\ll M^{(r-1)a}(M^{\rho b-a})^r(J_{s+r}(X/M^b))^{1-r/s}(I_{b,\rho b}(X))^{r/s}.$$
\end{lemma}

\begin{proof} In this instance, when $a\ge 1$ the conclusion asserted by the lemma follows from Lemma \ref{lemma5.2}. We therefore focus again on the situation in which $a=0$ and $b\in\{1,2\}$. We again find from (\ref{6.3}) and (\ref{6.4}) that the relation (\ref{6.14}) holds. In present circumstances, by imitating the argument of the proof of Lemma \ref{lemma5.2}, and substituting the application of Lemma \ref{lemma3.6} for our earlier use of Lemma \ref{lemma3.5}, with the discussion of the preamble to Lemma \ref{lemma6.2} in mind, we find that
$$K_{0,b}(X)\ll (M^{\rho b})^r(J_{s+r}(X/M^b))^{1-r/s}(I_{b,\rho b}(X))^{r/s}.$$
This yields the desired conclusion when $a=0$ and $b\in \{1,2\}$, and completes the proof of the lemma.
\end{proof}

\section{The efficient congruencing step, II} By means of Lemmata \ref{lemma6.2} and \ref{lemma6.3}, one is able to relate $K_{a,b}(X)$ either to $I_{b,kb}(X)$ or $I_{b,\rho b}(X)$, the purpose of the pre-congruencing step being to remove the constraint $a\ge 1$ imposed in \S5 so as to permit $a$ to be zero. In this section we complete the discussion of the efficient congruencing step by combining Lemma \ref{lemma4.3} first with Lemma \ref{lemma6.2}, and then with Lemma \ref{lemma6.3}, so as to obtain the basic iterative relations between $K_{a,b}(X)$ and $K_{b,kb+h}(X)$ in the first instance, and between $K_{a,b}(X)$ and $K_{b,\rho b+h}(X)$ in the second instance.

\begin{lemma}\label{lemma7.1} Define $s_0$ and $\kap$ as in (\ref{2.6}), and put $s=s_0$. Suppose that $a$ and $b$ are integers with $0\le a<b\le \frac{1}{3}(k\tet)^{-1}$, and put $H=2(k-1)b$ and $g=b-ka$. Suppose further that when $a=0$, one has $b=1$ or $2$. Then there exists an integer $h$, with $0\le h<H$, having the property that
\begin{align*}
\llbracket K_{a,b}(X)\rrbracket \ll &\, X^\del \left( M^{sg-(2s-r+1)h}\llbracket K_{b,kb+h}(X)\rrbracket \right)^{r/s}(X/M^b)^{\eta_{s+r}(1-r/s)}\\
&\, +M^{-rH/(3s)}(X/M^b)^{\eta_{s+r}}.
\end{align*}
\end{lemma}

\begin{proof} We assume throughout that $k\ge 3$ and $1\le r\le k-1$, so we may begin with the observation that $s=rk\ge \max\{3r,k-r+1\}$. Next, since $\frac{1}{2}r(r+1)\ge r$ for $r\ge 1$, we find from (\ref{2.6}) that
$$\kap=(rk-\tfrac{1}{2}r(r+1))\left( \frac{k+1}{k-1}\right) \le (rk-r)\left( \frac{k+1}{k-1}\right) =rk+r,$$
so that $\kap\le s+r$. On recalling (\ref{2.17}), we may therefore apply Lemma \ref{lemma6.2} to deduce that
\begin{equation}\label{7.1}
\llbracket K_{a,b}(X)\rrbracket \ll (M^b)^{2s}(M^a)^{2r-\kap}M^{\frac{1}{2}r(r-1)(b+a)}(M^{kb-a})^rT_1^{1-r/s}T_2^{r/s},
\end{equation}
where
$$T_1=\frac{J_{s+r}(X/M^b)}{X^{2s+2r-\kap}}\quad \text{and}\quad T_2=\frac{I_{b,kb}(X)}{X^{2s+2r-\kap}}.$$
But
\begin{equation}\label{7.2}
T_1\ll (M^{-b})^{2s+2r-\kap}(X/M^b)^{\eta_{s+r}+\del}.
\end{equation}
Writing $H=2(k-1)b$, we find that the hypotheses of the statement of the lemma guarantee that $kb+H=(3k-2)b<\tet^{-1}$. We therefore see from Lemma \ref{lemma4.3} that there exists an integer $h$ with $0\le h<H$ such that
$$T_2\ll \frac{M^{h(r-1)}K_{b,kb+h}(X)}{X^{2s+2r-\kap}}+\frac{M^{-(r+2)H/2}X^\del (X/M^b)^{\eta_{s+r}}}{(M^{kb})^{2s}(M^b)^{2r-\kap}}.$$
Fixing this value of $h$, we have
\begin{equation}\label{7.3}
T_2\ll (M^{-kb})^{2s}(M^{-b})^{2r-\kap}\Ome,
\end{equation}
where
$$\Ome =M^{-(2s-r+1)h}\llbracket K_{b,kb+h}(X)\rrbracket +M^{-(r+2)H/2}X^\del (X/M^b)^{\eta_{s+r}}.$$

\par On combining (\ref{7.1}), (\ref{7.2}) and (\ref{7.3}), we conclude that
$$\llbracket K_{a,b}(X)\rrbracket \ll M^{\ome (a,b)}(X/M^b)^{(1-r/s)(\eta_{s+r}+\del)}\Ome^{r/s},$$
where
\begin{align*}
\ome(a,b)=&\, 2sb+(2r-\kap )a+\tfrac{1}{2}r(r-1)(b+a)+r(kb-a)\\
&-(1-r/s)(2s+2r-\kap)b-(2skb+(2r-\kap)b)r/s.
\end{align*}
On recalling (\ref{2.6}), a brief computation reveals that
\begin{align*}
\ome(a,b)&=(\kap-rk+\tfrac{1}{2}r(r-1))b-(\kap-\tfrac{1}{2}r(r+1))a\\
&=\left(\frac{r(k-r)}{k-1}\right)(b-ka)=\left(\frac{r(k-r)}{k-1}\right)g.
\end{align*}
Consequently, on the one hand we have $\ome(a,b)\le rg$, and on the other
\begin{align*}
\ome(a,b)s/r-(r+1)H/2&\le \frac{s(k-r)b}{k-1}-(r+1)(k-1)b\\
&\le (kr-(r+1)(k-1))b=-(k-r-1)b\le 0.
\end{align*}
Thus we infer that
\begin{align*}
\llbracket K_{a,b}(X)\rrbracket \ll &\, (M^{-H/2})^{r/s}X^\del (X/M^b)^{\eta_{s+r}}\\
&\, +X^\del M^{rg-(2s-r+1)hr/s}(X/M^b)^{\eta_{s+r}(1-r/s)}\llbracket K_{b,kb+h}(X)\rrbracket^{r/s}.
\end{align*}
The conclusion of the lemma follows on observing that $\del$ may be assumed small enough that $X^\del\ll M^{rH/(6s)}$.
\end{proof}

\begin{lemma}\label{lemma7.2} Suppose that $1\le r\le \min\{k-2,\frac{1}{2}k+1\}$, define $s_0$ and $\kap$ as in (\ref{2.5}), and put $s=s_0$. Suppose that $a$ and $b$ are integers with $b\ge (r-1)a$ and $0\le a<b\le \frac{1}{3}(\rho \tet)^{-1}$, and put $H=2(\rho-1)b$ and $g=b-\rho a$. Suppose further that when $a=0$, one has $b=1$ or $2$. Then there exists an integer $h$, with $0\le h<H$, having the property that
\begin{align*}
\llbracket K_{a,b}(X)\rrbracket \ll &\, X^\del \left( M^{sg-(2s-r+1)h}\llbracket K_{b,\rho b+h}(X)\rrbracket \right)^{r/s}(X/M^b)^{\eta_{s+r}(1-r/s)}\\
&\, +M^{-rH/(3s)}(X/M^b)^{\eta_{s+r}}.
\end{align*}
\end{lemma}

\begin{proof} We now assume that $k\ge 3$ and $1\le r\le \min\{k-2,\frac{1}{2}k+1\}$, so we have
$$s=r\rho=r(k+1-r)\ge \max\{3r,k-r+1\}.$$
Next, we see from (\ref{2.5}) that
$$\kap=s_0+r-\frac{r-1}{k-r}\le s+r.$$
On recalling (\ref{2.17}), we may therefore apply Lemma \ref{lemma6.3} to deduce that
\begin{equation}\label{7.4}
\llbracket K_{a,b}(X)\rrbracket \ll (M^b)^{2s}(M^a)^{2r-\kap}M^{(r-1)a}(M^{\rho b-a})^rT_1^{1-r/s}T_2^{r/s},
\end{equation}
where in this instance
$$T_1=\frac{J_{s+r}(X/M^b)}{X^{2s+2r-\kap}}\quad \text{and}\quad T_2=\frac{I_{b,\rho b}(X)}{X^{2s+2r-\kap}}.$$
Writing $H=2(\rho-1)b$, the hypotheses of the statement of the lemma imply that $\rho b+H\le (3\rho-2)b<\tet^{-1}$. Thus we deduce from Lemma \ref{lemma4.3} that there exists an integer $h$ with $0\le h<H$ such that
$$T_2\ll \frac{M^{h(r-1)}K_{b,\rho b+h}(X)}{X^{2s+2r-\kap}}+\frac{M^{-(r+2)H/2}X^\del (X/M^b)^{\eta_{s+r}}}{(M^{\rho b})^{2s}(M^b)^{2r-\kap}}.$$
Fixing this value of $h$, we see that
\begin{equation}\label{7.5}
T_2\ll (M^{-\rho b})^{2s}(M^{-b})^{2r-\kap}\Ome ,
\end{equation}
where
$$\Ome =M^{-(2s-r+1)h}\llbracket K_{b,\rho b+h}(X)\rrbracket +M^{-(r+2)H/2}X^\del (X/M^b)^{\eta_{s+r}}.$$

\par On combining (\ref{7.4}) and (\ref{7.5}) with the estimate (\ref{7.2}), still valid in the present setting, we reach the upper bound
$$\llbracket K_{a,b}(X)\rrbracket \ll M^{\ome (a,b)}(X/M^b)^{(1-r/s)(\eta_{s+r}+\del)}\Ome^{r/s},$$
where
\begin{align*}
\ome(a,b)=&\, 2sb+(2r-\kap)a+(r-1)a+r(\rho b-a)\\
&-(1-r/s)(2s+2r-\kap)b-(2s\rho b+(2r-\kap)b)r/s.
\end{align*}
We next recall (\ref{2.5}), and hence deduce that
$$\ome (a,b)=(\kap -r\rho)b+(2r-\kap -1)a=\left( r-\frac{r-1}{\rho-1}\right) (b-\rho a).$$
Consequently, on the one hand we have $\ome(a,b)\le rg$, and on the other
\begin{align*}
\ome(a,b)s/r-(r+1)H/2&\le \rho\Bigl( r-\frac{r-1}{\rho-1}\Bigr) b-(r+1)(\rho-1)b\\
&\le (\rho r-r+1)b-(\rho r+\rho -r-1)b\\
&=-(k-r-1)b\le 0.
\end{align*}
Thus we conclude that
\begin{align*}
\llbracket K_{a,b}(X)\rrbracket \ll &\, (M^{-H/2})^{r/s}X^\del (X/M^b)^{\eta_{s+r}}\\
&+X^\del M^{rg-(2s-r+1)hr/s}(X/M^b)^{\eta_{s+r}(1-r/s)}\llbracket K_{b,\rho b+h}(X)\rrbracket^{r/s}.
\end{align*}
Just as in the conclusion of the proof of the previous lemma, our argument is completed by noting the estimate $X^\del\ll M^{rH/(6s)}$.
\end{proof}

\section{The iterative process, I: the basic estimate} Making use of Lemma \ref{lemma6.1}, and then applying either Lemma \ref{lemma7.1} repeatedly, or else Lemma \ref{lemma7.2} repeatedly, we are able to bound $J_{s+r}(X)$ in terms of quantities of the shape $K_{c,d}(X)$, wherein $c$ and $d$ pass through an increasing sequence of integral values. Our goal in this section is to control this iterative process so as to establish Theorems \ref{theorem1.1} and \ref{theorem1.4}. Although we model this treatment on the analogous analysis of \cite[\S7]{Woo2011a}, there are complications in the details that generate some complexity.

\begin{lemma}\label{lemma8.1} Define $s_0$ and $\kap$ as in (\ref{2.6}), and put $s=s_0$. Let $a$ and $b$ be integers with $0\le a<b\le \frac{1}{3}(k\tet)^{-1}$ having the property that when $a=0$, one has $b=1$ or $2$, and put $g=b-ka$. Suppose in addition that there exist non-negative numbers $\psi$, $c$ and $\gam$, with $c\le 3(s/r)^N$, for which
\begin{equation}\label{8.1}
X^{\eta_{s+r}(1+\psi \tet)}\ll X^{c\del}M^{-\gam}\llbracket K_{a,b}(X)\rrbracket .
\end{equation}
Then, for some non-negative integer $h$ with $h\le 2(k-1)b$, one has
$$X^{\eta_{s+r}(1+\psi'\tet)}\ll X^{c'\del}M^{-\gam'}\llbracket K_{a',b'}(X)\rrbracket ,$$
where
$$\psi'=(s/r)\psi +(s/r-1)b,\quad c'=(s/r)(c+1),$$
$$a'=b,\quad b'=kb+h,\quad \gam'=(s/r)\gam+(2s-r+1)h-sg.$$
\end{lemma}

\begin{proof} Since we may suppose that $c\le 3(s/r)^N$ and $\del<(Ns)^{-3N}$, we have $c\del <\tet/(6s)$, and hence $X^{c\del}<M^{1/(6s)}$. In addition, one has $M^{1/(6s)}>X^\del$. We therefore deduce from Lemma \ref{lemma7.1} that there exists an integer $h$ with $0\le h<2(k-1)b$ with the property that
\begin{align*}
\llbracket K_{a,b}(X)\rrbracket \ll &\, X^\del (X/M^b)^{(1-r/s)\eta_{s+r}}\left( M^{sg-(2s-r+1)h}\llbracket K_{b,kb+h}(X)\rrbracket\right)^{r/s}\\
&\, +M^{-r/(3s)}X^{\eta_{s+r}}.
\end{align*}
We are therefore led from the hypothesised bound (\ref{8.1}) to the estimate
\begin{align*}
X^{\eta_{s+r}(1+\psi \tet)}\ll &\, X^{(c+1)\del}M^{-\gam +rg-(2s-r+1)rh/s}(X/M^b)^{(1-r/s)\eta_{s+r}}\llbracket K_{b,kb+h}(X)\rrbracket^{r/s}\\
&\, +X^{\eta_{s+r}-\del},
\end{align*}
whence
$$X^{\eta_{s+r}(r/s+(\psi+(1-r/s)b)\tet)}\ll X^{(c+1)\del}M^{-\gam+rg-(2s-r+1)rh/s}\llbracket K_{b,kb+h}(X)\rrbracket^{r/s}.$$
The conclusion of the lemma follows on raising left and right hand sides in the last inequality to the power $s/r$.
\end{proof}

\begin{lemma}\label{lemma8.2} Define $s_0$ and $\kap$ as in (\ref{2.6}), and put $s=s_0$. Then $\eta_{s+r}=0$.
\end{lemma}

\begin{proof} We begin by recalling our convention concerning the value of $K_{0,b}(X)$ from the preamble to Lemma \ref{lemma6.2}. Thus, as a consequence of Lemma \ref{lemma6.1}, it follows from (\ref{2.16}) and (\ref{2.17}) that there exists an integer $h$ with $h\in\{0,1\}$ such that
$$\llbracket J_{s+r}(X)\rrbracket \ll M^{-(2s-r+1)h}\llbracket K_{0,1+h}(X)\rrbracket .$$
We therefore deduce from (\ref{2.18}) that, with $h=0$ or $1$, one has
\begin{equation}\label{8.2}
X^{\eta_{s+r}}\ll X^\del \llbracket J_{s+r}(X)\rrbracket \ll X^\del M^{-(2s-r+1)h}\llbracket K_{0,1+h}(X)\rrbracket .
\end{equation}

We may suppose that $\eta_{s+r}>0$, for otherwise there is nothing to prove. We next take $h_{-1}$ to be the integer $h$ for which the relation (\ref{8.2}) holds, and we define three sequences $(a_n)$, $(b_n)$, $(h_n)$ of non-negative integers for $0\le n\le N$ as follows. We put $a_0=0$ and $b_0=1+h_{-1}$. Then, when $0\le n<N$, we fix any integer $h_n$ with $0\le h_n\le 2(k-1)b_n$, and then define
\begin{equation}\label{8.3}
a_{n+1}=b_n\quad \text{and}\quad b_{n+1}=kb_n+h_n.
\end{equation}
Next we define the auxiliary sequences $(\psi_n)$, $(c_n)$, $(\gam_n)$ of non-negative real numbers for $0\le n\le N$ by putting $\psi_0=0$, $c_0=1$, $\gam_0=(2s-r+1)h_{-1}$. Then, for $0\le n<N$, we define
\begin{align}
\psi_{n+1}&=(s/r)\psi_n+(s/r-1)b_n,\label{8.4}\\
c_{n+1}&=(s/r)(c_n+1),\label{8.5}\\
\gam_{n+1}&=(s/r)\gam_n+(2s-r+1)h_n-sh_{n-1}.\label{8.6}
\end{align}
We note that a straightforward induction reveals $\gam_n$ to be non-negative for $n\ge 0$, for the relation (\ref{8.6}) yields the recurrence formula
\begin{align*}
\gam_{n+1}-(2s-r+1)h_n&=(s/r)(\gam_n-rh_{n-1})\\
&\ge (s/r)(\gam_n-(2s-r+1)h_{n-1}).
\end{align*}
On recalling that $s/r=k$, we therefore see that for $n\ge 1$ one has
\begin{align*}
\gam_n&\ge (2s-r+1)h_{n-1}+k^n(\gam_0-(2s-r+1)h_{-1})\\
&\ge (2s-r+1)h_{n-1}\ge 0,
\end{align*}
so that $\gam_n$ is indeed non-negative. A second induction confirms that for $0\le n\le N$, one has
$$c_n=\frac{2s-r}{s-r}\left( \frac{s}{r}\right)^n-\frac{s}{s-r}\le \left( 2+\frac{1}{k-1}\right)\left(\frac{s}{r}\right)^n\le 3(s/r)^n.$$

\par We claim that a choice may be made for the sequence $(h_n)$ in such a manner that for $0\le n\le N$, one has
\begin{equation}\label{8.7}
b_n<\sqrt{N}(s/r)^n
\end{equation}
and
\begin{equation}\label{8.8}
X^{\eta_{s+r}(1+\psi_n\tet)}\ll X^{c_n\del}M^{-\gam_n}\llbracket K_{a_n,b_n}(X)\rrbracket .
\end{equation}
When $n=0$, the relation (\ref{8.7}) holds by the definition of $b_0$. On the other hand, when $n=0$, the relation (\ref{8.8}) holds as a consequence of (\ref{8.2}). We initiate further analysis of larger indices $n$ with a preliminary discussion of the recurrence relations (\ref{8.3}) to (\ref{8.6}). Recall that $s=rk$, and observe that when $m\ge 1$, one has
$$\gam_{m+1}-(s/r)\gam_m=(2s-r+1)(b_{m+1}-kb_m)-s(b_m-kb_{m-1}),$$
whence
$$\gam_{m+1}-(2s-r+1)b_{m+1}+sb_m=k(\gam_m-(2s-r+1)b_m+sb_{m-1}).$$
It therefore follows by induction that for $m\ge 1$ one has
$$\gam_m\ge (2s-r+1)b_m-sb_{m-1}+k^{m-1}(\gam_1-(2s-r+1)b_1+sb_0).$$
We recall further that $b_0=1+h_{-1}$, $b_1=kb_0+h_0$, and so
\begin{align*}
\gam_1-(2s-r+1)b_1+sb_0=&\, (k\gam_0+(2s-r+1)h_0-sh_{-1})\\
&\, -(2s-r+1)(kb_0+h_0)+sb_0.\end{align*}
On recalling again the relation $s=rk$, we arrive at the formula
\begin{align*}
\gam_1-(2s-r+1)b_1+sb_0&=k(\gam_0-(2s-r+1)b_0)+s(b_0-h_{-1})\\
&=s-rk-k(2s-2r+1),
\end{align*}
and this in turn delivers the lower bound
\begin{equation}\label{8.9}
\gam_m\ge (2s-r+1)b_m-sb_{m-1}-(2s-2r+1)k^m.
\end{equation}

\par Suppose now that the desired conclusions (\ref{8.7}) and (\ref{8.8}) have been established for the index $n<N$. Then from (\ref{8.7}), one has $kb_n\tet <k(s/r)^{n-N-2}<\frac{1}{3}$, whence $b_n<\frac{1}{3}(k\tet)^{-1}$. We may therefore appeal to Lemma \ref{lemma8.1} to deduce from (\ref{8.8}) that there exists a non-negative integer $h$, with $h\le 2(k-1)b_n$, for which one has the upper bound
\begin{equation}\label{8.10}
X^{\eta_{s+r}(1+\psi'\tet)}\ll X^{c'\del}M^{-\gam'}\llbracket K_{a',b'}(X)\rrbracket ,
\end{equation}
where
\begin{align}
a'&=b_n=a_{n+1},\quad b'=kb_n+h,\label{8.11}\\
\psi'&=(s/r)\psi_n+(s/r-1)b_n=\psi_{n+1},\label{8.12}\\
c'&=(s/r)(c_n+1)=c_{n+1},\label{8.13}\\
\gam'&=(s/r)\gam_n+(2s-r+1)h-sh_{n-1}.\label{8.14}
\end{align}
Notice here that in the final relation (\ref{8.14}), we have made use of the formula $b_n-ka_n=b_n-kb_{n-1}=h_{n-1}$ available via (\ref{8.3}).\par

Suppose, if possible, that $b'\ge \sqrt{N}(s/r)^{n+1}=\sqrt{N}k^{n+1}$. The relations (\ref{8.11}) and (\ref{8.14}) together with (\ref{8.9}) show that
\begin{align}
\gam'&=(s/r)\gam_n+(2s-r+1)(b'-kb_n)-s(b_n-kb_{n-1})\notag \\
&=k(\gam_n-(2s-r+1)b_n+sb_{n-1})+(2s-r+1)b'-sb_n\notag \\
&\ge -(2s-2r+1)k^{n+1}+(2s-2r+1)b'+r(b'-kb_n)\notag \\
&\ge (2s-2r+1)(b'-k^{n+1})\ge (1-1/\sqrt{N})(2s-2r+1)b'.\label{8.15}
\end{align}
But $b'=kb_n+h\le (3k-2)b_n<\tet^{-1}$, and so it follows from Lemma \ref{lemma5.3} that
\begin{equation}\label{8.16}
\llbracket K_{a',b'}(X)\rrbracket \ll X^{\eta_{s+r}+\del}(M^{b'})^\kap .
\end{equation}
Combining (\ref{8.15}), (\ref{8.16}) and (\ref{8.10}), therefore, we obtain the bound
\begin{equation}\label{8.17}
X^{\eta_{s+r}(1+\psi_{n+1}\tet)}\ll X^{\eta_{s+r}+(c_{n+1}+1)\del }(M^{b'})^{\kap -(2s-2r+1)(1-1/\sqrt{N})}.
\end{equation}
We now recall that $c_{n+1}\le 3(s/r)^{n+1}$, so that $X^{(c_{n+1}+1)\del}<M^{1/2}$. Also, when $r\ge 1$ and $k\ge 3$ one has
\begin{align*}
\kap-&(1-1/\sqrt{N})(2s-2r+1)\\
&\le (rk-\tfrac{1}{2}r(r+1))\left( \frac{k+1}{k-1}\right) -2rk+2r-1+2s/\sqrt{N}\\
&\le (rk-r)\left( \frac{k+1}{k-1}\right) +(2-2k)r-\tfrac{1}{2}\\
&=r(k+1)+(2-2k)r-\tfrac{1}{2}=(3-k)r-\tfrac{1}{2}\le -\tfrac{1}{2}.
\end{align*}
Thus we obtain
\begin{equation}\label{8.18}
X^{\eta_{s+r}(1+\psi_{n+1}\tet)}\ll X^{\eta_{s+r}}M^{(1-b')/2}\ll X^{\eta_{s+r}}M^{-1/2}.
\end{equation}
Since $\psi_{n+1}$ and $\tet$ are both positive, we are forced to conclude that $\eta_{s+r}<0$, contradicting our opening hypothesis. The assumption that $b'\ge \sqrt{N}(s/r)^{n+1}$ is therefore untenable, and so we must in 
fact have $b'<\sqrt{N}(s/r)^{n+1}$. We take $h_n$ to be the integer $h$ at hand, so that $b'=b_{n+1}$ and $\gam'=\gam_{n+1}$, and thereby we obtain the desired conclusion that (\ref{8.7}) and (\ref{8.8}) hold with 
$n$ replaced by $n+1$. This completes the present inductive step.\par

We have confirmed the validity of (\ref{8.7}) and (\ref{8.8}) for $0\le n\le N$. We have also the bounds $c_n\le 3(s/r)^n$, $\gam_n\ge 0$ and $b_n\ge k^n$. Furthermore, since $s=rk$ one finds that
$$\psi_{n+1}=k\psi_n+(k-1)b_n\ge k\psi_n+(k-1)k^n,$$
whence $\psi_n\ge n(k-1)k^{n-1}$. Finally, one has $b_N\tet<(r/s)^2<1$, so that $b_N<\tet^{-1}$. An application of Lemma \ref{lemma5.3} in combination with (\ref{8.8}) therefore delivers the estimate
$$X^{\eta_{s+r}(1+\psi_N\tet)}\ll X^{\eta_{s+r}+(c_N+1)\del}(M^{b_N})^\kap \ll X^{\eta_{s+r}+k^2}.$$
Again making use of the relation $\tet=N^{-1/2}(r/s)^{N+2}$ recorded in (\ref{2.9}), we thus obtain the estimate
$$\eta_{s+r}\le \frac{k^2}{\psi_N\tet}\le \frac{\sqrt{N}k^2(s/r)^{N+2}}{N(k-1)k^{N-1}}<\frac{k^5}{\sqrt{N}}.$$
We are at liberty to take $N$ as large as we please in terms of $k$, and thus $\eta_{s+r}$ can be made arbitrarily small. It follows that $\eta_{s+r}=0$, and this completes the proof of the lemma.
\end{proof}

The conclusion of Theorem \ref{theorem1.4} is an immediate consequence of Lemma \ref{lemma8.2}. The latter shows that when $s\ge r(k+1)$, one has
$$J_s(X)\ll X^{2s-\kap+\eps},$$
where
$$\kap=(rk-\tfrac{1}{2}r(r+1))\left( \frac{k+1}{k-1}\right) .$$
Write $t=k-r$. Then this estimate may be rewritten to state that when $s\ge (k+1)(k-t)$, one has
$$J_s(X)\ll X^{2s-\frac{1}{2}k(k+1)+\Del_s+\eps},$$
where
\begin{align*}
\Del_s&=\tfrac{1}{2}k(k+1)-\left((k-t)k-\tfrac{1}{2}(k-t)(k-t+1)\right)\left( \frac{k+1}{k-1}\right)\\
&=\tfrac{1}{2}t(t-1)\left( \frac{k+1}{k-1}\right) .
\end{align*}
This completes the proof of Theorem \ref{theorem1.4} for $1\le t\le k-1$. The special case in which $t=1$ delivers the exponent $\Del_s=0$, so that when $s\ge k^2-1$ one has
$$J_s(X)\ll X^{2s-\frac{1}{2}k(k+1)+\eps}.$$
The conclusion of Theorem \ref{theorem1.1} therefore follows as a speical case of Theorem \ref{theorem1.4}.

\section{The iterative process, II: quasi-diagonal behaviour} Our handling of the iterative process must be modified in order to establish Theorem \ref{theorem1.2}, though the strategy is very similar to that underlying the proof of Theorem \ref{theorem1.4}. There are sufficiently many differences from the treatment presented in \S8 that, in the interests of enhancing clarity, we provide a fairly complete account in this section.

\begin{lemma}\label{lemma9.1} Suppose that $1\le r\le \min\{k-2,\frac{1}{2}k+1\}$, define $s_0$ and $\kap$ as in (\ref{2.5}), and put $s=s_0$. Let $a$ and $b$ be integers with $b\ge (r-1)a$ and $0\le a<b\le \frac{1}{3}(\rho \tet)^{-1}$ having the property that when $a=0$, one has $b=1$ or $2$, and put $g=b-\rho a$. Suppose in addition that there exist non-negative numbers $\psi$, $c$ and $\gam$, with $c\le 3(s/\rho)^N$, for which
\begin{equation}\label{9.1}
X^{\eta_{s+r}(1+\psi \tet)}\ll X^{c\del}M^{-\gam}\llbracket K_{a,b}(X)\rrbracket .
\end{equation}
Then, for some non-negative integer $h$ with $h\le 2(\rho-1)b$, one has
$$X^{\eta_{s+r}(1+\psi'\tet)}\ll X^{c'\del}M^{-\gam'}\llbracket K_{a',b'}(X)\rrbracket ,$$
where
$$\psi'=(s/r)\psi +(s/r-1)b,\quad c'=(s/r)(c+1),$$
$$a'=b,\quad b'=\rho b+h,\quad \gam'=(s/r)\gam+(2s-r+1)h-sg.$$
\end{lemma}

\begin{proof} We follow the argument of the proof of Lemma \ref{lemma8.1}, noting first that $X^{c\del}<M^{1/(6s)}$ and $M^{1/(6s)}>X^\del$. Then from Lemma \ref{lemma7.2} there exists an integer $h$ with $0\le h\le 2(\rho-1)b$ with the property that
\begin{align*}
\llbracket K_{a,b}(X)\rrbracket \ll &\, X^\del (X/M^b)^{(1-r/s)\eta_{s+r}}\left(M^{sg-(2s-r+1)h}\llbracket K_{b,\rho b+h}(X)\rrbracket\right)^{r/s}\\
&\, +M^{-r/(3s)}X^{\eta_{s+r}}.
\end{align*}
The hypothesised bound (\ref{9.1}) therefore implies that
\begin{align*}
X^{\eta_{s+r}(1+\psi \tet)}&\ll X^{(c+1)\del}M^{-\gam +rg-(2s-r+1)rh/s}(X/M^b)^{(1-r/s)\eta_{s+r}}\llbracket K_{b,\rho b+h}(X)\rrbracket^{r/s}\\
&\, +X^{\eta_{s+r}-\del},
\end{align*}
whence
$$X^{\eta_{s+r}(r/s+(\psi+(1-r/s)b\tet))}\ll X^{(c+1)\del}M^{-\gam +rg-(2s-r+1)rh/s}\llbracket K_{b,\rho b+h}(X)\rrbracket^{r/s}.$$
The conclusion of the lemma follows.
\end{proof}

\begin{lemma}\label{lemma9.2} Let $r$ be a natural number with $1\le r\le \min\{k-2,\frac{1}{2}k+1\}$. Define $s_0$ and $\kap$ as in (\ref{2.5}), and put $s=s_0$. Then $\eta_{s+r}=0$.
\end{lemma}

\begin{proof} We follow the proof of Lemma \ref{lemma8.2}, supposing that $\eta_{s+r}>0$. We begin by observing that the discussion of the first paragraph of the proof of Lemma \ref{lemma8.2} remains valid in the present circumstances, and so we may take $h_{-1}$ to be an integer $h$ for which the relation (\ref{8.2}) holds. In this instance we define the sequences $(a_n)$, $(b_n)$, $(h_n)$ of non-negative integers for $0\le n\le N$ as follows. We put $a_0=0$ and $b_0=1+h_{-1}$. Then, when $0\le n<N$, we fix any integer $h_n$ with $0\le h_n\le 2(\rho-1)b_n$, and then define
\begin{equation}\label{9.2}
a_{n+1}=b_n\quad \text{and}\quad b_{n+1}=\rho b_n+h_n.
\end{equation}
The auxiliary sequences $(\psi_n)$, $(c_n)$, $(\gam_n)$ of non-negative real numbers are defined for $0\le n\le N$ by putting $\psi_0=0$, $c_0=1$, $\gam_0=(2s-r+1)h_{-1}$. Then for $0\le n<N$, we define $\psi_{n+1}$, $c_{n+1}$, $\gam_{n+1}$ in terms of $\psi_n$, $c_n$, $\gam_n$ by means of the respective formulae (\ref{8.4}), (\ref{8.5}) and (\ref{8.6}). We note that a straightforward induction again reveals $\gam_n$ to be non-negative for $n\ge 0$, just as in the proof of Lemma \ref{lemma8.2}. One has $s/r=\rho$, and hence one finds that
\begin{align*}
\gam_n&\ge (2s-r+1)h_{n-1}+\rho^n(\gam_0-(2s-r+1)h_{-1})\\
&\ge (2s-r+1)h_{n-1}\ge 0.
\end{align*}
We also have $c_n\le 3(s/r)^n$.\par

We claim that a choice may be made for the sequence $(h_n)$ in such a manner that for $0\le n\le N$, one has the upper bounds (\ref{8.7}) and (\ref{8.8}). As in our earlier discussion, these estimates hold for $n=0$ as a consequence of the definition of $b_0$ together with (\ref{8.2}). A comparison of the relations (\ref{9.2}) and (\ref{8.3}) reveals that the only adjustment necessary is to switch $k$ in (\ref{8.3}) to $\rho$ in (\ref{9.2}), though in present circumstances one has $s/r=\rho$. Thus we find as in the argument leading to (\ref{8.9}) that in the present situation, one has for $m\ge 1$ that
\begin{equation}\label{9.3}
\gam_m\ge (2s-r+1)b_m-sb_{m-1}-(2s-2r+1)\rho^m.
\end{equation}

\par Suppose now that the desired conclusions (\ref{8.7}) and (\ref{8.8}) have been established for the index $n<N$. Then one has $\rho b_n\tet<\rho (s/r)^{n-N-2}<\frac{1}{3}$, whence $b_n<\frac{1}{3}(\rho \tet)^{-1}$. Also, our hypotheses on $r$ ensure that
$$b_n\ge \rho b_{n-1}=(k-r+1)a_n\ge (r-1)a_n.$$
An application of Lemma \ref{lemma9.1} therefore leads from (\ref{8.8}) to the conclusion that there exists an integer $h$, with $h\le 2(\rho-1)b_n$, for which one has the upper bound (\ref{8.10}), where $a'$, $\psi'$, $c'$, $\gam'$ satisfy (\ref{8.11})--(\ref{8.14}), and in addition
\begin{equation}\label{9.4}
b'=\rho b_n+h.
\end{equation}

\par Suppose, if possible, that $b'\ge \sqrt{N}(s/r)^{n+1}=\sqrt{N}\rho^{n+1}$. Then as in the argument of the proof of Lemma \ref{lemma8.2} leading to (\ref{8.15}) above, we find that (\ref{8.14}) and (\ref{9.4}) together with (\ref{9.3}) show that
\begin{align}\gam'&=(s/r)\gam_n+(2s-r+1)(b'-\rho b_n)-s(b_n-\rho b_{n-1})\notag \\
&\ge (2s-2r+1)(b'-\rho^{n+1})\ge (1-1/\sqrt{N})(2s-2r+1)b'.\label{9.5}
\end{align}
But $b'=\rho b_n+h\le (3\rho-2)b_n<\tet^{-1}$, and so it follows from Lemma \ref{lemma5.3} that (\ref{8.16}) holds. Combining (\ref{9.5}), (\ref{8.16}) and (\ref{8.10}), therefore, we obtain the bound (\ref{8.17}). Observe next that in present circumstances, one deduces from (\ref{2.5}) that
\begin{align*}
\kap-(1-1/\sqrt{N})(2s-2r+1)&\le s+r-\frac{r-1}{k-r}-2s+2r-1+\frac{2s}{\sqrt{N}}\\
&<3r-\rho r-\tfrac{1}{2}=(r+2-k)r-\tfrac{1}{2}.
\end{align*}
Since, by assumption, we have $r\le k-2$, it follows that
$$\kap-(1-1/\sqrt{N})(2s-2r+1)\le -\tfrac{1}{2},$$
and thus we obtain again the relation (\ref{8.18}). From here, one deduces as before that $\eta_{s+r}<0$, contradicting our opening hypothesis, and leading us to conclude that in fact $b'<\sqrt{N}(s/r)^{n+1}$. We take $h_n$ to be the integer $h$ at hand, so that $b'=b_{n+1}$ and $\gam'=\gam_{n+1}$, and thereby deduce that (\ref{8.7}) and (\ref{8.8}) hold with $n$ replaced by $n+1$. This completes the proof of the present inductive step.\par

Next, since (\ref{8.7}) and (\ref{8.8}) both hold for $0\le n\le N$, one has $b_N\tet<(r/s)^2<1$, so that $b_N<\tet^{-1}$. From (\ref{9.2}) one has $b_n\ge \rho^n$. Since $s=r\rho$, one finds that
$$\psi_{n+1}=\rho \psi_n+(\rho-1)b_n\ge \rho \psi_n+(\rho -1)\rho^n,$$
so that $\psi_n\ge n(\rho-1)\rho^{n-1}$. An application of Lemma \ref{lemma5.3} therefore leads from (\ref{8.8}) to the upper bound
$$X^{\eta_{s+r}(1+\psi_N\tet)}\ll X^{\eta_{s+r}+(c_N+1)\del}(M^{b_N})^\kap\ll X^{\eta_{s+r}+k^2}.$$
But from (\ref{2.9}) we have $\tet=N^{-1/2}(r/s)^{N+2}$, and thus
$$\eta_{s+r}\le \frac{k^2}{\psi_N\tet}\le \frac{\sqrt{N}k^2(s/r)^{N+2}}{N(\rho-1)\rho^{N-1}}<\frac{k^2\rho^3}{\sqrt{N}}.$$
On taking $N$ sufficiently large in terms of $k$, we are able to make $\eta_{s+r}$ as small as we please. It follows that $\eta_{s+r}=0$, and this completes the proof of the lemma.
\end{proof}

The conclusion of Theorem \ref{theorem1.2} follows from Lemma \ref{lemma9.2}. The latter shows that when $t=\rho r+r=r(k-r+2)$, then one has
$$J_t(X)\ll X^{2t-(t-(r-1)/(k-r))+\eps}=X^{t+\nu_t+\eps},$$
in which $\nu_t=(r-1)/(k-r)$. When $s\le t$, meanwhile, one may apply H\"older's inequality to obtain
\begin{align*}
J_s(X)&=\oint |f_k(\bfalp;X)|^{2s}\d\bfalp \le \Bigl( \oint |f_k(\bfalp;X)|^{2t}\d\bfalp \Bigr)^{s/t}\\
&\ll (X^{t+\nu_t+\eps})^{s/t}\ll X^{s+\nu_t+\eps}.
\end{align*}
This completes the proof of Theorem \ref{theorem1.2} for $1\le r\le \min\{k-2,\frac{1}{2}k+1\}$.\par

We observe that when $k\ge 4$, the hypotheses of the statement of Theorem \ref{theorem1.2} are satisfied with $r=[(k+1)/2]$. In such circumstances, when $k=2l+1$ is odd, one has
$$r(k-r+2)=(l+1)(l+2)\ge (l+\tfrac{1}{2})^2+2l+1=\tfrac{1}{4}k^2+k,$$
and when $k=2l$ is even, one has
$$r(k-r+2)=l(l+2)=\tfrac{1}{4}k^2+k.$$
Meanwhile, one may easily verify that in each case the exponent $\nu_{r,k}$ satisfies
$$\nu_{r,k}=\frac{r-1}{k-r}\le 1.$$
The conclusion of Corollary \ref{corollary1.3} therefore follows directly from Theorem \ref{theorem1.2}.\par

Finally, suppose that $2\le r\le \min\{k-2,\tfrac{1}{2}k+1\}$, and put $t(r)=r(k-r+2)$. Then whenever $t(r-1)\le s\le t(r)$, it is a consequence of Theorem \ref{theorem1.2} that $J_{s,k}(X)\ll X^{s+\nu+\eps}$, where
$$\nu=\frac{r-1}{k-r}\le \frac{t(r-1)}{(k-r)(k-r+3)}\le \frac{4s}{k^2}.$$
Thus we see that the upper bound (\ref{1.3}) does indeed hold with a permissible exponent $\del_{s,k}$ satisfying $\del_{s,k}=O(s/k^2)$, thereby justifying the discussion following the statement of Theorem \ref{theorem1.2}.

\section{The asymptotic formula in Waring's problem} Our first application of the improved mean value estimate supplied by Theorem \ref{theorem1.1} concerns the asymptotic formula in Waring's problem. In this context, we define the exponential sum $g(\alp)=g_k(\alp;X)$ by
$$g_k(\alp;X)=\sum_{1\le x\le X}e(\alp x^k).$$
Also, we define the set of minor arcs $\grm=\grm_k$ to be the set of real numbers $\alp\in [0,1)$ satisfying the property that, whenever $a\in \dbZ$ and $q\in \dbN$ satisfy $(a,q)=1$ and $|q\alp-a|\le (2k)^{-1}X^{1-k}$, then $q>(2k)^{-1}X$. We begin by applying the methods of \cite{Woo2011b} to derive a mean value estimate restricted to minor arcs.

\begin{theorem}\label{theorem10.1} Suppose that $s\ge k^2-1$. Then for each $\eps>0$, one has
$$\int_\grm |g_k(\alp;X)|^{2s}\d\alp \ll X^{2s-k-1+\eps}.$$
\end{theorem}

\begin{proof} According to \cite[Theorem 2.1]{Woo2011b}, one has
$$\int_\grm|g_k(\alp;X)|^{2s}\d\alp \ll X^{\frac{1}{2}k(k-1)-1}(\log X)^{2s+1}J_{s,k}(2X).$$
Theorem \ref{theorem1.1} shows that when $s\ge k^2-1$, one has $J_{s,k}(2X)\ll X^{2s-\frac{1}{2}k(k+1)+\eps}$, and the conclusion of the theorem now follows.
\end{proof}

We transform the estimate supplied by this theorem into a less strident bound useful in handling the minor arc contribution in Waring's problem. For each natural number $k$, define the positive integer $s_0(j)=s_0(k,j)$ by means of the relation
$$s_0(k,j)=2k^2-2k-\frac{2(k-1)(j+1)-2^{j+1}}{k-j}.$$
We then put
\begin{equation}\label{10.1}
s_1(k)=\min_{\substack{0\le j\le k-2\\ 2^j\le k^2-k-1}}s_0(k,j).
\end{equation}

\begin{lemma}\label{lemma10.2} Suppose that $k$ is a natural number with $k\ge 3$. Then
$$\int_0^1|g_k(\alp;X)|^{s_1(k)}\d\alp \ll X^{s_1(k)-k+\eps}.$$
Moreover, when $s$ is a real number with $s>s_1(k)$, there exists a positive number $\del=\del(k,s)$ with the property that
$$\int_\grm|g_k(\alp;X)|^s\d\alp \ll X^{s-k-\del}.$$
\end{lemma}

\begin{proof} The second estimate claimed in the lemma is immediate from Theorem \ref{theorem10.1} when $s\ge 2k^2-2$, on making use of the trivial estimate $|g_k(\alp;X)|\le X$. We suppose therefore that $s_1(k)<s\le 2k^2-2$, and we put $\tau=s-s_1(k)$. Let $j$ be an integer with $0\le j\le k-2$ and $2^j\le k(k-1)-1$ for which $s_1(k)=s_0(k,j)$. Then by H\"older's inequality, one has
$$\int_\grm|g(\alp)|^s\d\alp\le \Bigl( \int_\grm |g(\alp)|^{2k^2-2}\d\alp \Bigr)^a\Bigl( \int_0^1|g(\alp)|^{2^{j+1}}\d\alp \Bigr)^b,$$
where
$$a=\frac{s-2^{j+1}}{2k^2-2-2^{j+1}}\quad \text{and}\quad b=\frac{2k^2-2-s}{2k^2-2-2^{j+1}}.$$
An application of Theorem \ref{theorem10.1} in combination with Hua's lemma (see \cite[Lemma 2.5]{Vau1997}) therefore yields the bound
\begin{align*}
\int_\grm|g(\alp)|^s\d\alp &\ll X^\eps (X^{(2k^2-2)-k-1})^a(X^{2^{j+1}-j-1})^b\\
&\ll X^{s-k-\nu+\eps},
\end{align*}
where $\nu=a-(k-j-1)b$. A modicum of computation reveals that
$$\nu=\frac{(k-j)(s-s_1(k))}{2k^2-2-2^{j+1}}\ge \tau/(2k^2),$$
and so the second conclusion of the lemma therefore follows with $\del=\tau/(4k^2)$.\par

When $s=s_1(k)$, the above discussion shows that
\begin{equation}\label{10.2}
\int_\grm |g(\alp)|^s\d\alp \ll X^{s-k+\eps}.
\end{equation}
But on writing $\grM=[0,1)\setminus \grm$, the methods of \cite[Chapter 4]{Vau1997} confirm that whenever $s\ge k+2$, one has
$$\int_\grM |g(\alp)|^s\d\alp \ll X^{s-k}.$$
The first conclusion of the lemma follows by combining this estimate with the earlier bound (\ref{10.2}).
\end{proof}

The argument following the proof of \cite[Lemma 3.1]{Woo2011b} may now be adapted, without effort, to show that $\Gtil(k)\le [s_1(k)]+1$ for $k\ge 3$. The first conclusion of Theorem \ref{theorem1.5} consequently follows at once from the definition (\ref{10.1}). This upper bound for $\Gtil(k)$ is easily made explicit for smaller values of $k$. Thus, on taking $r=3$, one finds that for $k\ge 5$ one has
$$\frac{2(k-1)(r+1)-2^{r+1}}{k-r}=\frac{8k-24}{k-3}=8,$$
and on taking $r=4$, one finds that for $k\ge 6$ one has
$$\frac{2(k-1)(r+1)-2^{r+1}}{k-r}=\frac{10k-42}{k-4}=10-\frac{2}{k-4},$$
which is at least $9$ for $k\ge 6$, and exceeds $9$ for $k\ge 7$. Also, on taking $r=5$, one finds that
$$\frac{2(k-1)(r+1)-2^{r+1}}{k-r}=\frac{12k-76}{k-5}=12-\frac{16}{k-5},$$
a quantity which exceeds $10$ for $k\ge 14$. Thus we deduce that
$$\Gtil(6)\le 52,\quad \Gtil(k)\le 2k^2-2k-9\quad (7\le k\le 13)$$
and
$$\Gtil(k)\le 2k^2-2k-10\quad (k\ge 14).$$

\par An alternative to the above approach proceeds by means of the methods of Ford \cite{For1995}. Motivated by the notation introduced in (\ref{2.16}), we write
$$\llbracket J_{t,k}(Y)\rrbracket^*=Y^{\frac{1}{2}k(k+1)-2t}J_{t,k}(Y).$$
One may then rephrase \cite[Theorem 1]{For1995} in the following form.

\begin{theorem}\label{theorem10.3}
Let $m$ be an integer with $1\le m\le k$. Then for each natural number $s$ with $s\ge \frac{1}{2}m(m-1)$, one has
$$\int_0^1|g_k(\alp;X)|^{2s}\d\alp \ll X^{2s-k}\llbracket J_{s-\frac{1}{2}m(m-1),k}(X^{1/m})\rrbracket^*.$$
\end{theorem}

For each natural number $k$, we now consider integers $m$ and $t$ with $1\le m\le k$ and $1\le t\le k-1$, and we define $\Del_{t,k}$ as in (\ref{1.4}). We then put
$$s_2(k,m,t)=2k^2-2-\frac{2(t-1)(k+1)-m(m-1)}{1+\Del_{t,k}/m},$$
and set
$$s_3(k)=\underset{2(t-1)(k+1)+m(m-1)<2k^2-2}{\min_{1\le m\le k}\min_{1\le t\le k-1}}s_2(k,m,t).$$

\begin{lemma}\label{lemma10.4}
Suppose that $s$ and $k$ are natural numbers with $k\ge 3$ and $s>s_3(k)$. Then there exists a positive number $\del=\del(k,s)$ with the property that
$$\int_\grm |g_k(\alp;X)|^s\d\alp \ll X^{s-k-\del}.$$
\end{lemma}

\begin{proof} As in the proof of Lemma \ref{lemma10.2}, the desired conclusion is immediate from Theorem \ref{theorem10.1} when $s\ge 2k^2-2$, on making use of the trivial estimate $|g_k(\alp;X)|\le X$. We suppose therefore that $s_3(k)<s\le 2k^2-2$, and we put $\tau=s-s_3(k)$. Let $m$ and $t$ be integers with $1\le m\le k$, $1\le t\le k-1$ and $2(t-1)(k+1)+m(m-1)<2k^2-2$, for which $s_3(k)=s_2(k,m,t)$. Then by H\"older's inequality, one has
$$\int_\grm|g(\alp)|^s\d\alp\le \Bigl( \int_\grm |g(\alp)|^{2k^2-2}\d\alp \Bigr)^a\Bigl( \int_0^1|g(\alp)|^{2(k-t)(k+1)+m(m-1)}\d\alp \Bigr)^b,$$
where
$$a=\frac{s-2(k-t)(k+1)-m(m-1)}{2k^2-2-2(k-t)(k+1)-m(m-1)}$$
and
$$b=\frac{2k^2-2-s}{2k^2-2-2(k-t)(k+1)-m(m-1)}.$$

\par By applying Theorem \ref{theorem10.3} and Theorem \ref{theorem1.4} in sequence, one finds that
$$\int_0^1|g(\alp)|^{2(k-t)(k+1)+m(m-1)}\d\alp \ll X^{2(k-t)(k+1)+m(m-1)-k+\Del_{t,k}/m+\eps}.$$
Consequently, an application of Theorem \ref{theorem10.1} yields the bound
\begin{align*}
\int_\grm|g(\alp)|^s\d\alp &\ll X^\eps (X^{(2k^2-2)-k-1})^a(X^{2(k-t)(k+1)+m(m-1)-k+\Del_{t,k}/m})^b\\
&\ll X^{s-k-\nu+\eps},
\end{align*}
where
$$\nu=a-b\Del_{t,k}/m=\frac{\left(1+\Del_{t,k}/m\right)(s-s_3(k))}{2k^2-2-2(k-t)(k+1)-m(m-1)}\ge \tau/(2k^2).$$
The conclusion of the lemma therefore follows with $\del=\tau/(4k^2)$.
\end{proof}

The argument following the proof of \cite[Lemma 3.1]{Woo2011b} may again be adapted to show that $\Gtil(k)\le [s_3(k)]+1$ for $k\ge 3$. One can check by means of a direct computation that when $k=20$, if one takes $t=7$ and $m=9$, then $s_2(k,m,t)<748$, and in this way one obtains the bound $\Gtil(20)\le 748$. In view of the discussion following the proof of Lemma \ref{lemma10.2}, this completes the proof of Corollary \ref{corollary1.7}. Similarly, the conclusion of Corollary \ref{corollary1.6} follows on taking $t=2[k^{1/3}]$ and $m=[k^{2/3}]$, for then one finds that
$$\Del_{t,k}/m=\frac{\frac{1}{2}t(t-1)}{m}\left( \frac{k+1}{k-1}\right) =\frac{2k^{2/3}+O(k^{1/3})}{k^{2/3}+O(1)}=2+O(k^{-1/3}),$$
and hence
\begin{align*}s_2(k,m,t)&=2k^2-2-\frac{2k(2k^{1/3})-k^{4/3}+O(k)}{3+O(k^{-1/3})}\\
&=2k^2-k^{4/3}+O(k).
\end{align*}

\par We finish by noting that the proof of \cite[Theorem 4.2]{Woo2011b} may be adapted transparently so as to establish that when $s>\min\{ s_1(k),s_3(k)\}$, then the anticipated asymptotic formula holds for the number of integral solutions of the diagonal equation
$$a_1x_1^k+\ldots +a_sx_s^k=0,$$
with $|\bfx|\le B$. Here, the coefficients $a_i$ $(1\le i\le s)$ are fixed integers. Similar improvements may be wrought in upper bounds for $\Gtil^+(k)$, the least number of variables required to establish that the anticipated asymptotic formula in Waring's problem holds for almost all natural numbers $n$. Thus, one may adapt the methods of \cite[\S5]{Woo2011b} to show that
$$\Gtil^+(k)\le k^2-k+1-\max_{\substack{0\le j\le k-2\\ 2^j\le k^2-k-1}}\left\lceil \frac{(k-1)(j+1)-2^j}{k-j}\right\rceil $$
and
$$\Gtil^+(k)\le k^2-\underset{2(t-1)(k+1)+m(m-1)<2k^2-2}{\max_{1\le m\le k}\max_{1\le t\le k-1}}\left\lceil \frac{(t-1)(k+1)-\frac{1}{2}m(m-1)}{1+\Del_{t,k}/m}\right\rceil .$$

\section{Further applications} In this section we briefly discuss some applications of the mean value estimates supplied by Theorems \ref{theorem1.1} and \ref{theorem1.4}, with the aim of noting improvements made available over our previous work \cite{Woo2011a}. We begin with an analogue of Weyl's inequality.

\begin{theorem}\label{theorem11.1} Let $k$ be an integer with $k\ge 4$, and let $\bfalp\in \dbR^k$. Suppose that there exists a natural number $j$ with $2\le j\le k$ such that, for some $a\in \dbZ$ and $q\in \dbN$ with $(a,q)=1$, one has $|\alp_j-a/q|\le q^{-2}$ and $q\le X^j$. Then one has
$$f_k(\bfalp;X)\ll X^{1+\eps}(q^{-1}+X^{-1}+qX^{-j})^{\sig(k)},$$
where $\sig(k)^{-1}=2k(k-2)$.
\end{theorem}

\begin{proof} Under the hypotheses of the statement of the theorem, we find that \cite[Theorem 5.2]{Vau1997} shows that for $s\in \dbN$, one has
$$f_k(\bfalp;X)\ll (J_{s,k-1}(2X)X^{\frac{1}{2}k(k-1)}(q^{-1}+X^{-1}+qX^{-j}))^{1/(2s)}\log (2X).$$
The conclusion of the theorem therefore follows on taking
$$s=(k-1)^2-1=k(k-2),$$
for in such circumstances Theorem \ref{theorem1.1} delivers the bound
$$J_{s,k-1}(2X)\ll X^{2s-\frac{1}{2}k(k-1)+\eps}.$$
\end{proof}

The proof of \cite[Theorem 1.6]{Woo2011a} may be easily adapted to deliver estimates depending on common diophantine approximations.

\begin{theorem}\label{theorem11.2} Let $k$ be an integer with $k\ge 4$, and let $\tau$ and $\del$ be real numbers with $\tau^{-1}>4k(k-2)$ and $\del>k\tau$. Suppose that $X$ is sufficiently large in terms of $k$, $\del$ and $\tau$, and further that $|f_k(\bfalp;X)|>X^{1-\tau}$. Then there exist integers $q$, $a_1,\ldots,a_k$ such that $1\le q\le X^\del$ and $|q\alp_j-a_j|\le X^{\del-j}$ $(1\le j\le k)$.\end{theorem}

The proof of \cite[Theorem 1.7]{Woo2011a} likewise delivers the following result concerning the distribution modulo $1$ of polynomial sequences. Here, we write $\|\tet\|$ for ${\min}_{y\in \dbZ}|\tet-y|$.

\begin{theorem}\label{theorem11.3} Let $k$ be an integer with $k\ge 4$, and define $\tau(k)$ by $\tau(k)^{-1}=4k(k-2)$. Then whenever $\bfalp\in \dbR^k$ and $N$ is sufficiently large in terms of $k$ and $\eps$, one has
$$\min_{1\le n\le N}\|\alp_1n+\alp_2n^2+\ldots +\alp_kn^k\|<N^{\eps-\tau(k)}.$$
\end{theorem}

In each of Theorems \ref{theorem11.2} and \ref{theorem11.3}, the exponent $4k(k-2)$ represents an improvement on the exponent $4k(k-1)$ made available in \cite[Theorems 1.6 and 1.7]{Woo2011a}. In \cite[Theorem 1.5]{Woo2011a}, meanwhile, we established a conclusion similar to that of Theorem \ref{theorem11.1}, though with a weaker exponent $\sig(k)$ satisfying $\sig(k)^{-1}=2k(k-1)$. As with this earlier work, our estimates supersede the Weyl exponent $\sig(k)=2^{1-k}$ when $k\ge 8$, and supersede work of Heath-Brown \cite{HB1988} and Robert and Sargos \cite{RS2000} for $k\ge 9$. When $k=8$, in fact, our exponent matches that of Heath-Brown \cite{HB1988}, though our conclusion is applicable for a substantially larger set of coefficients.\par

We turn next to Tarry's problem. When $h$, $k$ and $s$ are positive integers with $h\ge 2$, consider the Diophantine system
\begin{equation}\label{11.1}
\sum_{i=1}^sx_{i1}^j=\sum_{i=1}^sx_{i2}^j=\ldots =\sum_{i=1}^sx_{ih}^j\quad (1\le j\le k).
\end{equation}
Let $W(k,h)$ denote the least natural number $s$ having the property that the simultaneous equations (\ref{11.1}) possess an integral solution $\bfx$ with
$$\sum_{i=1}^sx_{iu}^{k+1}\ne \sum_{i=1}^sx_{iv}^{k+1}\quad (1\le u<v\le h).$$

\begin{theorem}\label{theorem11.4} When $h$ and $k$ are natural numbers with $h\ge 2$ and $k\ge 2$, one has $W(k,h)\le k^2-\sqrt{2}k^{3/2}+4k$.\end{theorem}

\begin{proof} The argument of the proof of \cite[Theorem 1.3]{Woo2011a} shows that $W(k,h)\le s$ whenever $J_{s,k+1}(X)=o(J_{s,k}(X))$. Incorporating the bounds for $J_{s,k+1}(X)$ supplied via Theorem \ref{theorem1.4} into this argument, one finds that
$$W(k,h)\le (k+1-t)(k+2)$$
whenever
$$2s-\tfrac{1}{2}(k+1)(k+2)+\tfrac{1}{2}t(t-1)(1+2/k)<2s-\tfrac{1}{2}k(k+1),$$
a constraint equivalent to the condition
$$t(t-1)<\frac{2k(k+1)}{k+2}=2k-2+\frac{4}{k+2}.$$
By direct computation, one finds that this inequality is satisfied when $t=[\sqrt{2k}]$, but not for $t\ge \sqrt{2k}+1$. Thus we deduce that
\begin{align*}
W(k,h)&\le (k-[\sqrt{2k}]+1)(k+2)=k^2+3k-(k+2)[\sqrt{2k}]+2\\
&\le k^2-\sqrt{2}k^{3/2}+4k+4-2\sqrt{2k}.
\end{align*}
The conclusion of the theorem follows immediately.
\end{proof}

In \cite[Theorem 1.3]{Woo2011a}, we obtained the weaker bound $W(k,h)\le k^2+k-2$. We remark that the conclusion of Theorem \ref{theorem11.4} may be utilised to obtain an improvement in a result of Croot and Hart related to the sum-product theorem. When $A$ is a set of real numbers, write
$$A\cdot A=\{xy:\text{$x\in A$ and $y\in A$}\}$$
and
$$hA=\{x_1+\ldots +x_h:\text{$x_i\in A$ $(1\le i\le h)$}\}.$$

\begin{theorem}\label{theorem11.5}
Suppose that $h$ and $n$ are natural numbers with $h\ge 2$. Let $A$ be a set of $n$ real numbers. Then whenever $\eps$ is a positive number sufficiently small in terms of $h$, and $|A\cdot A|\le n^{1+\eps}$, there exists a positive number $\lam$ having the property that
$$|h(A\cdot A)|>n^{\lam h^{1/3}}.$$
\end{theorem}

The aforementioned result of Croot and Hart (see \cite[Theorem 1.2]{CH2010}) delivers a similar conclusion, though with the exponent $h^{1/3}$ replaced by $(h/\log h)^{1/3}$.\par

\par We note also that on writing
$$\grS(s,k)=\sum_{q=1}^\infty \underset{(a_1,\ldots ,a_k,q)=1}{\sum_{a_1=1}^q\dots \sum_{a_k=1}^q}\Bigl|q^{-1}\sum_{r=1}^qe((a_1r+\ldots +a_kr^k)/q)\Bigr|^{2s}$$
and
$$\calJ(s,k)=\int_{\dbR^k}\Bigl| \int_0^1e(\bet_1\gam+\ldots +\bet_k\gam^k)\d\gam \Bigr|^{2s}\d\bfbet ,$$
the method of proof of \cite[Theorem 1.2]{Woo2011a} may be modified in the light of Theorem \ref{theorem1.1} to obtain the asymptotic formula
$$J_{s,k}(X)\sim \grS(s,k)\calJ(s,k)X^{2s-\frac{1}{2}k(k+1)},$$
provided only that $k\ge 3$ and $s\ge k^2$. In \cite[Theorem 1.2]{Woo2011a}, such a conclusion was obtained for $s\ge k^2+k+1$. A similar improvement holds also for work on the asymptotic formula in the Hilbert-Kamke problem.\par

Finally, write
$$F_k(\bfbet;X)=\sum_{1\le x\le X}e(\bet_kx^k+\bet_{k-2}x^{k-2}+\ldots +\bet_1x).$$
L.-K. Hua investigated the problem of bounding the least integer $C_k$ such that, whenever $s\ge C_k$, one has
$$\oint |f_k(\bfalp;X)|^s\d\bfalp \ll X^{s-\frac{1}{2}k(k+1)+\eps},$$
and likewise the least integer $S_k$ such that, whenever $s\ge S_k$, one has
$$\oint |F_k(\bfbet;X)|^s\d\bfbet \ll X^{s-\frac{1}{2}(k^2-k+2)+\eps}.$$

\begin{theorem}\label{theorem11.6} When $k\ge 3$, one has $C_k\le 2k^2-2$ and $S_k\le 2k^2-2k$.
\end{theorem}

\begin{proof} The bound on $C_k$ is immediate from Theorem \ref{theorem1.1}. In order to establish the bound on $S_k$, we begin by observing that \cite[equation (10.10)]{Woo2011a} supplies the estimate
\begin{equation}\label{11.2}
\oint |F_k(\bfbet;X)|^{2t}\d\bfbet \ll X^{k-2+\eps}J_{t,k}(2X)+X^{\eps-1}J_{t,k-1}(2X).
\end{equation}
Write $u=(k-2)(k+1)$. Then an application of Theorem \ref{theorem1.4} with $t=2$ shows that
$$J_{u,k}(2X)\ll X^{2u-\frac{1}{2}k(k+1)+\Del},$$
with $\Del=(k+1)/(k-1)$. Consequently, on applying H\"older's inequality in combination with Theorem \ref{theorem1.1}, we obtain the bound
\begin{align*}
J_{k(k-1),k}(X)&\le \Bigl( \oint |f_k(\bfalp;X)|^{2u}\d\bfalp\Bigr)^{(k-1)/(k+1)}\Bigl( \oint |f_k(\bfalp;X)|^{2k^2-2}\d\bfalp\Bigr)^{2/(k+1)}\\
&\ll X^\eps (X^{2u-\frac{1}{2}k(k+1)+(k+1)/(k-1)})^{(k-1)/(k+1)}(X^{2k^2-2-\frac{1}{2}k(k+1)})^{2/(k+1)}\\
&\ll X^{2k(k-1)-\frac{1}{2}k(k+1)+1+\eps}.
\end{align*}
On the other hand, it follows from Theorem \ref{theorem1.1} that whenever $s\ge k(k-2)$, then one has
$$J_{s,k-1}(X)\ll X^{2s-\frac{1}{2}k(k-1)+\eps}.$$
On substituting these estimates into (\ref{11.2}), we conclude that
\begin{align*}
\oint |F_k(\bfbet;X)|^{2k(k-1)}\d\bfbet &\ll X^{2k(k-1)+\eps}(X^{1-\frac{1}{2}k(k+1)+(k-2)}+X^{-\frac{1}{2}k(k-1)-1})\\
&\ll X^{2k(k-1)-\frac{1}{2}(k^2-k+2)+\eps}.
\end{align*}
We therefore see that $S_k\le 2k(k-1)$, and this completes the proof of the theorem.
\end{proof}

For comparison, in \cite[Theorems 1.1 and 10.3]{Woo2011a} we derived the weaker bounds $C_k\le 2k^2+2k$ and $S_k\le 2k^2+2k-4$. When $k\ge 4$, the conclusion of Theorem \ref{theorem11.6} improves also on the bounds obtained by Hua \cite[Chapter 5]{Hua1965}, namely
$$C_3\le 16,\quad C_4\le 46,\quad C_5\le 110,\ldots $$
and
$$S_3\le 10,\quad S_4\le 32,\quad S_5\le 86,\ldots .$$
Moreover, Theorem \ref{theorem11.6} matches the bound established by Hua for $C_3$.

\bibliographystyle{amsbracket}
\providecommand{\bysame}{\leavevmode\hbox to3em{\hrulefill}\thinspace}

\end{document}